\documentclass[11pt]{article}

\usepackage{amsmath,amssymb,amsthm}
\usepackage{float}
\usepackage[utf8]{inputenc}
\usepackage[english]{babel}
\usepackage[T1]{fontenc}
\usepackage{lmodern}
\usepackage[totalwidth=158mm, totalheight=255mm]{geometry}

\usepackage[auth-sc, affil-it]{authblk}
\usepackage{tikz}
\usepackage{tikz-3dplot}
\usetikzlibrary{calc}
\usepackage{verbatim}

\usepackage{pstricks}
\usepackage{pgf}
\usepackage{curve2e}

\newcommand{\com}[1]{\fbox{\textsf{#1}}}

\renewcommand{\geq}{\geqslant}

\newcommand{\bbC}{\mathbb{C}}

\newcommand{\Kc}{\mathcal{K}}

\newcommand{\Dc}{\mathcal{D}}

\newcommand{\Mc}{\mathfrak{M}}

\newcommand{\M}{\mathcal{M}}

\newcommand{\Rc}{\mathcal{R}}

\theoremstyle{plain}
\newtheorem{main-theorem}{Theorem}
\newtheorem{theo}{Theorem}[section]

\newtheorem{lemma}[theo]{Lemma}
\newtheorem{coro}[theo]{Corollary}
\newtheorem{deff}[theo]{Definition}
\newtheorem{exam}[theo]{Example}
\newtheorem{rema}[theo]{Remark}

\newtheorem{defi}[main-theorem]{Definition}

\newtheorem{Thm}{Theorem}

\newtheorem*{clai-nn}{Claim}
\newtheorem*{theorem*}{Theorem}
\newtheorem{Qu}[Thm]{Question}
\newtheorem*{schoencond*}{Sch\"onflies Condition}

\author[1]{Jun Luo\thanks{Supported by the Chinese National Natural Science Foundation Projects 11771391.}}
\author[1,2]{Yi Yang}
\author[1]{Xiao-Ting Yao}

\affil[1]{School of Mathematics,
    Sun Yat-sen University, Guangzhou 510275, China
}

\affil[2]{Corresponding author: yangyi32@mail2.sysu.edu.cn}

\title{\Large Peano Model for Planar Compacta and a Lemma by Beardon}
\date{\today}

\begin{document}
\maketitle

\begin{abstract}
A  \emph{Peano space} is a compactum that has at most countably many non-degenerate components $\{P_n\}$, which are locally connected, such that the sequence of diameters $\{\delta(P_n):n\}$ either is finite  or converges to zero.
Given a compactum $K\subset\hat{\mathbb{C}}$.
It is known that, among all  the monotone decompositions of $K$ with Peano hyperspaces, there exists a unique one, denoted as $\Dc_K^{PS}$, that is finer than all the others. We call $\Dc_K^{PS}$  the \emph{core decomposition of $K$} with Peano hyperspace. The resulted hyperspace under quotient topology, still denoted as $\Dc_K^{PS}$, is called the {\em Peano model for $K$}. We show that $\Dc_K^{PS}$ is independent of the embedding of $K$ into $\hat{\mathbb{C}}$, in the sense that $\Dc_{h(K)}^{PS}=\left\{h(d): d\in\Dc_K^{PS}\right\}$ for any other embedding $h: K\rightarrow\hat{\mathbb{C}}$. Given a rational function $f$ with $deg(f)\ge2$ that is independent of $K$. A well known result by Beardon says that the pre-image $f^{-1}(d)$ for any $d\in\Dc_K^{PS}$ has $l\le deg(f)$ components $N_1, \ldots, N_l$ satisfying $f(N_1)=\cdots=f(N_l)=d$. We show that $N_1,\ldots, N_l\in \Dc_{f^{-1}(K)}^{PS}$. Therefore the induced map $\tilde{f}(d)=f(d)$ is continuous and solves Question 5.4 proposed by Curry (MR2642461) and extends earlier partial results by Blokh-Curry-Oversteegen (MR2737795 and MR3008890), when $K$  is assumed to be unshielded and $f$ a polynomial with $f^{-1}(K)=K=f(K)$. 
Finally, we introduce a function $\lambda_K: K\rightarrow\mathbb{N}$, such that $\lambda_K(x)\equiv0$ if and only if $K$ is a Peano space. This function and its maximum are topologically invariant, while its level set $\lambda_K^{-1}(0)$ at zero is of particular interest. For instance, when $K$ is the Mandelbrot set $\lambda_K^{-1}(0)$ contains all the hyperbolic components except possibly the roots of the primitive ones.


\end{abstract}

\newpage

{\footnotesize \tableofcontents}

\section{Introduction and the Main Theorems}

In this paper a compactum $K$ means a compact metric space and we mostly consider compacta in Euclidean spaces. The topology of such a compactum $K$ can be extremely complicated, even if $K$ is on the plane. To analyse certain aspects of the topology of $K$, one may turn to explore an appropriate decomposition $\Dc$ of $K$ such that the hyperspace keeps the most basic features of $K$ and is itself a member from a special family of topological spaces, whose properties are more or less well understood.
A very early example of such an analysis comes from
Moore's fundamental result: if $\Dc$ is a monotone decomposition of $\hat{\mathbb{C}}$ into non-separating continua the hyperspace is homeomorphic to $\hat{\mathbb{C}}$. Following the spirit of Moore's work we will focus on monotone decompositions of a compactum $K$. Motivated by recent studies on polynomial and rational Julia sets \cite{BCO11,BCO13,Curry10,Kiwi04,LLY-2017}  we further require the hyperspace to be a \emph{Peano space}, that is, a compactum whose non-degenerate components are locally connected and form a null sequence, in the sense that for any constant $C>0$ at most finitely many of them are of diameter greater than $C$. Clearly, a toally disconnected compactum is a Peano space. And a Peano continuum is just a connected Peano space.

Let $\Mc^{PS}(K)$ be the collection of all the monotone decompositions of a compactum $K$ with Peano hyperspaces. The member of $\Mc^{PS}(K)$ that is finer than all the others, if it exists, will be called the {\em core decomposition of $K$ with Peano hyperspace}, denoted as $\Dc_K^{PS}$.
When the core decomposition of $K$ exists, the hyperspace $\Dc_K^{PS}$ under quotient topology is called the {\em Peano model of $K$} and is still denoted as $\Dc_K^{PS}$.

The  sphere $\hat{\mathbb{C}}$ seems to be a reasonable {\em ambient space} for the compactum under consideration, if we want to study its Peano model. Actually, a compactum $K\subset\mathbb{R}$ is already a Peano space, while the Peano model for a compactum $K\subset\mathbb{R}^3$ may not exist. Example \ref{no-CD-ex} constructs a concrete continuum in $\mathbb{R}^3$ which is semi-locally connected everywhere but does not have a Peano model. By \cite[Theorem 7]{LLY-2017} the Peano model for a compactum $K\subset\mathbb{R}^2$ always exists. Here we note that the core decomposition $\Dc_K^{PS}$ of  a planar continuum $K$ may be strictly finer than the decomposition proposed by Moore \cite{Moore25-a}, whose hyperspace is also a Peano contiuum. In Example \ref{NLC-point-CD} we concretly construct such a continuum.

To fit with the setting of rational functions, we focus on compacta $K\subset\hat{\mathbb{C}}$ in the sequel.

For an unshielded compactum $K\subset\hat{\mathbb{C}}$ the core decomposition $\Dc_K^{PS}$ has been discussed  by Blokh-Curry-Oversteegen \cite{BCO11,BCO13}. When such a compactum $K$ is connected the existence of $\Dc_K^{PS}$ is obtained in \cite{BCO11}, in which the hyperspace  $\Dc_K^{PS}$ is called the finest locally connected model; when $K$ is disconnected the existence  $\Dc_K^{PS}$ is obtained in \cite{BCO13} and the corresponding hyperspace is called the finest finitely suslinian monotone model. These models are themselves planar and unshielded compacta \cite[Theorem 19]{BCO13}, and the establishment of them in \cite{BCO11,BCO13} are directly based on Moore's fundamental result from \cite{Moore25}.

It is known that every locally connected or finitely suslinian compactum $K\subset\hat{\mathbb{C}}$ is a Peano space \cite[Theorems 1 and 3]{LLY-2017}. Therefore, the Peano model obtained in \cite[Theorem 7]{LLY-2017} extends the two models in \cite{BCO11,BCO13} on unshielded case to the case of all compacta $K\subset\hat{\mathbb{C}}$. This extension does not use Moore's result and provides a negative answer to \cite[Question 5.2]{Curry10}: {\em Does there exist a rational function whose Julia set does not have a finest locally connected model ?} Moreover, it even resolves the first part of \cite[Question 5.4]{Curry10}: {\em For what useful topological properties $P$ does there exist a finest decomposition of every Julia set $J(R)$ {\rm(of a rational function $R$)} satisfying $P$ ? Is the decomposition dynamic? Which of these is the appropriate analogue for the finest locally connected model ? }

The current paper provides a definite answer to the latter two parts of \cite[Question 5.4]{Curry10} and discusses two fundamental questions: (1) does the Peano model of $K\subset\hat{\mathbb{C}}$ depend on the embedding $h: K\rightarrow\hat{\mathbb{C}}$ ? (2) how are the two decompositions $\Dc_K^{PS}$ and $\Dc_{f^{-1}(K)}^{PS}$ connected for an arbitrary rational function $f$ ? The first question is of its own interest from a topological viewpoint, and leads to the introduction of a new topological invariant. The second has direct motivations from the study of complex dynamics; moreover, it is closely related to  \cite[p.141, Factor Theorem (4.1)]{Whyburn42}, which states that if $A$ is a compactum and $\phi(A)=B$ a continuous map there exists a unique factorization $\phi(x)=\phi_2\circ\phi_1(x)$ such that $\phi_1(A)=A'$ is monotone and $\phi_2(A')=B$ is light. Here a continuous map is {\em monotone} if the point inverses are connected, and it is {\em light} if the point inverses are totally disconnected. The representation $\phi=\phi_2\circ \phi_1$ will be called the {\bf monotone-light factorization} of the continuous surjection $\phi: A\rightarrow B$. In our situation $B$ is chosen to be an arbitrary compactum $K\subset\hat{\mathbb{C}}$ and $A=f^{-1}(K)$ for an arbitrary rational function $f$; then we will determine the functions $\phi_1,\phi_2$ in the monotone factorization of the composite $\phi(x)=\pi\circ f(x)$ of the projection $\pi: K\rightarrow\Dc_K^{PS}$ with the restriction of $f$ to $f^{-1}(K)$. Here we note that $\pi$ is monotone and $f$ is light, so that $\phi=\pi\circ f$ is a {\bf light-monotone factorization}.

The core decomposition $\Dc_K^{PS}$ is helpful in understanding the topology of $K$, especially when $K$ is the connected Julia set of a polynomial $f$. In such a case, the continuum $K$ is located on the boundary of the (unbounded) component $U_K$ of $\hat{\mathbb{C}}\setminus K$ that contains $\infty$. The domain $U_K$  is conformally isomorphic to $\mathbb{D}^*=\left\{z\in\hat{\mathbb{C}}: |z|>1\right\}$.  By \cite[Lemma 16 to 17]{BCO11} the core decomposition $\Dc_K^{PS}$ is exactly the finest monotone decomposition whose elements are unions of impressions of $U_K$.
We want to mention two topics of some interest, which arise naturally in the above setting and have origins from recent studies on the dynamics of complex polynomials.

The first is motivated by Kiwi's fundamental results \cite[Theorems 2 to 3]{Kiwi04}. In terms of cored decomposition those results are recalled as follows: {\em If $K$ is the Julia set of a polynomial $f$ has no irrationally indifferent cycles then every element of the core decomposition $\Dc_K^{PS}$ is the union of finitely many prime end impressions; in particular, if $x\in K$ is a periodic or pre-periodic point of $f$ then the element of $\Dc_K^{PS}$ containing $x$ is just the singleton $\{x\}$.} More precisely, by \cite[Theorem 3]{Kiwi04} the element $\pi(x)$ of $\Dc_K^{PS}$ containing $x\in K$ equals ${\rm Fiber}(x)$. Here,
by \cite[Definition 2.5]{Kiwi04}, the set ${\rm Fiber}(x)$ consists of all the points $y\in K$ such that $x$ and $y$ lie in the same component of $K\setminus Z$ for any finite subset $Z$ consisting of periodic or pre-periodic points of $f$ with $\{x,y\}\cap Z=\emptyset$. Among others, we propose the analysis of planar continua $K$ with connected complement $\hat{\mathbb{C}}\setminus K$ that satisfy the following `FI' property:
\begin{center}
\com{\em every element $d$ of $\Dc_K^{PS}$ with $d\cap\partial K\ne\emptyset$ is made up of finitely many prime end impressions.}
\end{center}
Douady's work \cite{Douady93} ensures that the Mandelbrot set satisfies  the FI property. Naturally, one may ask the above question when $K$ is a multicorn.

The second topic is more or less connected to the existence of `sheer component' in the interior of the Mandelbrot set $\M$ (the conjecture is that such a component does not exist). We may adopt the term `ghost component', whose meaning is related to the core decomposition $\Dc_K^{PS}$ of a {\em full} compactum $K\subset\hat{\mathbb{C}}$, such that the  complement $U_K=\hat{\mathbb{C}}\setminus K$ is connected. More concretely, we follow Blokh-Curry-Oversteegen \cite{BCO11, BCO13} to derive a monotone decomposition of $\hat{\mathbb{C}}$:
\[\Dc_K^*=\Dc_K^{PS}\cup\left\{\{z\}: z\in U_K\right\}.\]
Then the hyperspace of $\Dc_K^*$ is a special type of {\em mantoid} \cite[p.27, \S 6.1]{Youngs51}, called {\em cactoid}. From Blokh-Curry-Oversteegen's discussions in \cite{BCO11,BCO13}
one can infer that the previous cactoid contains a sphere $S$ with $\pi(U_K)\subset S$, where $\pi: \hat{\mathbb{C}}\rightarrow\Dc_K^*$ is  the natural projection. If a component $G$ of $K^o$ satisfies $\pi(G)\cap S=\emptyset$ we call it  a {\em ghost component}. And we propose the question below:
\begin{center}
\com{\em Under what conditions does a compactum $K\subset\hat{\mathbb{C}}$ (with or without FI) have a ghost component?}
\end{center}


The first main result of this paper is closely related to \cite[Theorem 30]{BCO11} and \cite[Theorem 20]{BCO13}. In terms of core decompositions, those two theorems may be combined into a unified version as follows: {\em if a compactum $K$ is completely invariant under a polynomial $f$ then $f(d)\in\Dc_K^{PS}$ for each $d\in \Dc_K^{PS}$}. This will be included as a special case of the following extension, which is based on a thorough analysis on how the elements of  $\Dc_{f^{-1}(K)}^{PS}$ are changed under an arbitrary rational function $f$, that is independent of $K$. Therefore, we have found a definite answer to the second and third parts of \cite[Question 5.4]{Curry10}.

\begin{Thm}[\bf Invariance]\label{invariance}
If $K\subset\hat{\mathbb{C}}$ is a compactum and $f:\hat{\mathbb{C}}\rightarrow\hat{\mathbb{C}}$ is a rational function then  $f(d)\in\Dc_K^{PS}$ for each $d\in \Dc_{f^{-1}(K)}^{PS}$. Consequently, we have $f^{-1}\left(\Dc_K^{PS}\right)=\Dc_{f^{-1}(K)}^{PS}$, where $f^{-1}\left(\Dc_K^{PS}\right)$ is the collection of all those components of $f^{-1}(d)$ for  $d$ running through $\Dc_K^{PS}$.
\end{Thm}

The above property of the core decomposition with Peano hyperspace will be referred to as the {\em invariance under rational functions}. We want to mention several issues concerning this invariance that are more noteworthy than others.

Firstly,  Theorem \ref{invariance} extends the more restricted cases discussed in \cite[Theorem 30]{BCO11} and \cite[Theorem 20]{BCO13}. The major difference comes from three facts: (1) $K$ is arbitrary in Theorem \ref{invariance} and is assumed to be unshielded in \cite{BCO11,BCO13}; (2) $K$ is independent of the rational function $f$ in Theorem \ref{invariance} and is completely invariant under $f$  in \cite{BCO11,BCO13}, while $f$ is assumed to be a polynomial. On the other hand, Theorem \ref{invariance} is also related to  \cite[p.141, Factor Theorem (4.1)]{Whyburn42}. Under the conditions of Theorem \ref{invariance} we may choose $\phi=\pi\circ f: f^{-1}(K)\rightarrow\Dc_K^{PS}$ and apply the Factor Theorem to deduce the uniquely determined maps $\phi_1: f^{-1}(K)\rightarrow A'$ and  $\phi_2: A'\rightarrow\Dc_K^{PS}$. The ending statement of Theorem \ref{invariance} then implies that the space $A'$ is homeomorphic with $\Dc_{f^{-1}(K)}^{PS}$. Therefore, the monotone map $\phi_1$ is essentially  the natural projection $\pi: f^{-1}(K)\rightarrow\Dc_{f^{-1}(K)}^{PS}$ and the light map $\phi_2$ is given by $\phi_2(d)=f(d)\in\Dc_K^{PS}$ for $d\in\Dc_{f^{-1}(K)}^{PS}$.

Secondly, the former part of Theorem \ref{invariance} does not hold if $f^{-1}(K)$ is replaced by a compactum $L$ with $f(L)=K$. See Example \ref{invariance-destroyed} for two concrete compacta $K,L\subset\mathbb{C}$ and a polynomial $f$ with $f(L)=K$ such that the image of some element of $\Dc_L^{PS}$ is the union of uncountably many elements of $\Dc_K^{PS}$. Moreover, if $K$ is a continuum the core decomposition $\Dc_K^{PS}$ in Theorem \ref{invariance} can not be replaced by $\Dc_K^{SLC}$, the core decomposition of $K$ with semi-locally connected hyperspace. See Example \ref{slc-bad}. However, this does not solve \cite[Question 5.3]{Curry10}.

Thirdly, in the special case that $K$ is the Julia set of $f$, the ending statement of Theorem \ref{invariance} reads as $f^{-1}\left(\Dc_K^{PS}\right)=\Dc_K^{PS}$. By setting $\tilde{f}(d)=f(d)$ we obtain a factor system $\tilde{f}:\Dc_K^{PS}\rightarrow\Dc_K^{PS}$ of the restriction $f: K\rightarrow K$. This provides a very helpful approach to understand the topology of the Julia set $K$ of a rational function $f$ and the dynamics of $f$ restricted to its Julia set, especially when $f$ is a polynomial. An immediate observation is that the factor system $\tilde{f}:\Dc_K^{PS}\rightarrow\Dc_K^{PS}$ is either trivial (when $\Dc_K^{PS}$ has a single element) or topologically mixing  (when $\Dc_K^{PS}$ has more than one hence uncountably many  elements). Further discussions on the case of non-trivial $\tilde{f}:\Dc_K^{PS}\rightarrow\Dc_K^{PS}$ are expected.

Lastly but not least importantly, a primary motivation for Theorem \ref{invariance} also comes from a well known result by Beardon \cite{Beardon-91}, which appears as Lemma 5.7.2 in \cite[p.95]{Beardon91}. Actually, we need to cite this result in the process of proving Theorem \ref{invariance}. As a direct corollary of Theorem \ref{invariance}, we may strengthen Beardon's result as follows.

\begin{Thm}[\bf Beardon's Lemma]\label{Beardon}
If $M\subset\hat{\mathbb{C}}$ is a continuum and $f:\hat{\mathbb{C}}\rightarrow\hat{\mathbb{C}}$ a rational function with $deg(f)\ge2$ then the pre-image $f^{-1}(M)$ has $l\le deg(f)$ components $N_1,\ldots,N_l$ with $f(N_1)=\cdots=f(N_l)=M$. If further $M$ is an element of the core decomposition $\Dc_K^{PS}$ for a compactum $K\subset\hat{\mathbb{C}}$ then every $N_i$ is an element of the core decomposition $\Dc_{f^{-1}(K)}^{PS}$.
\end{Thm}

The Peano model $\Dc_K^{PS}$ is  based on the {\em Sch\"onflies relation $R_K$} \cite[Definition 4]{LLY-2017}, the closedness of which as a subset of $K^2$ is still not known. In deed, our proof for Theorem \ref{invariance} starts from a characterization for the closure of $R_K$. See Theorem \ref{closure} in Section \ref{A}. Such a characterization leads to the definition of the $S$-function, an analogue of the $T$-function studied in \cite{FitzGerald67}. The $S$-function only depends on the topology of $K$ and may be defined on any compactum, planar or non-plananr. With the help of this characterization we can prove.

\begin{Thm}\label{embedding}
If $h: K\rightarrow\hat{\mathbb{C}}$ is an embedding of a compactum  $K\subset\hat{\mathbb{C}}$ then $h(d)\in\Dc_{h(K)}^{PS}$ for every $d\in\Dc_K^{PS}$. Therefore,  the Peano models $\Dc_K^{PS}$ and $\Dc_{h(K)}^{PS}$ are homeomorphic.
\end{Thm}

The result of Theorem \ref{embedding} prepares the path to introduce a new topological invariant for compacta in the plane. The basic philosophy to define such an invariant is to consider the elements of $\Dc_K^{PS}$ as ``basic units'', which may be called {\em atoms} of $K$, and then analyze the hierarchy made up of
``atoms of atoms''. More precisely, for any compactum $K\subset\hat{\mathbb{C}}$ we can define a function $\lambda_K: K\rightarrow\mathbb{N}$ in the following way.

Set $N_0=K$. If there exist an integer $m\ge0$ and a (strictly) decreasing sequence of continua $N_1\supset N_2\supset\cdots\supset N_{m+1}=\{x\}$ such that $N_{i+1}\in \Dc_{N_i}^{PS}$ for $0\le i\le m$, we set $\lambda(x)=m$; otherwise, we set $\lambda(x)=\infty$.

Let $\lambda(K)=\sup\{\lambda(x): x\in K\}$. Theorem \ref{embedding} then indicates that $\lambda(K)=\lambda(h(K))$ for any homeomorphism $h: K\rightarrow L$. In other words, the quantity $\lambda(K)$ is a topological invariant among planar compacta. Clearly, it can be very difficult to compute $\lambda(K)$ for a general compactum $K\subset\hat{\mathbb{C}}$, if we know little about its topology. On the other hand, for specific choices of $K$,  the level set $\lambda_K^{-1}(0)$ can be very interesting, and is ``computable'' in some sense. See examples in Section \ref{scale-examples}.

A non-empty level set $\lambda_K^{-1}(0)$ plays an important role in analyzing the topology of $K$. Firstly, it is a $G_\delta$-set; secondly, if $K$ is the Julia set of a rational function $f$ then Theorem \ref{Beardon} implies that $\lambda_K^{-1}(0)$ is completely invariant under $f$. Therefore, when $\mu$ is an ergodic measure of the system $f: K\rightarrow K$ the measure $\mu\left(\lambda_K^{-1}(0)\right)$ is either $0$ or $1$. And, in the case that $K$ is the Mandelbrot set, the following questions are open and seem to be of some interest:

\begin{Qu}\label{D}
Let $\M$ be the Mandelbrot set. Is it true that $\lambda_\M^{-1}(0)$ contains all the parameters $c\in\partial \M$  such that $z\mapsto z^2+c$ is not infinitely renormalizable ? Is it true that all the hyperbolic components are contained in a single component of $\lambda_\M^{-1}(0)$ ? Is it true that $\lambda(\M)\le 1$ ? Let $\pi: \M\rightarrow\Dc_\M^{PS}$ be the natural projection of $\M$ onto its Peano model. Is $\pi\left(\M\setminus\lambda_K^{-1}(0)\right)$  totally disconnected ?
\end{Qu}

The above questions are mostly motivated by well known fundamental properties of $\M$. Some of those properties are closely related to the Peano model $\Dc_\M^{PS}$ and largely improve our understanding of $\lambda_\M^{-1}(0)$, and $\M$ itself. In particular, we have the following.
\begin{Thm}\label{E}
If $H$ is a hyperbolic component of $\M^o$ then every point $x\in\overline{H}$ belongs to $\lambda_\M^{-1}(0)$, except for the only cases when $H$ is primitive and $x$ is the root of $H$. However, $\frac14\in\lambda_\M^{-1}(0)$.
\end{Thm}

Given a continuum $K$ in the plane, the level set $\lambda_K^{-1}(0)$ may not equal the set of points $x\in K$ such that $K$ is locally connected at $x$. For instance, one may choose $K$ to be the union of $[0,1]\subset\mathbb{C}$ and the Cantor product $\{u+iv: u\in\mathcal{C}, 0\le v\le1\}$, where $\mathcal{C}$ is Cantor's ternary set. Then $K$ is locally connected at $x\in K$ is if and only if $x\in[0,1]$, while $\lambda_K^{-1}(0)=[0,1]\setminus\mathcal{C}$. In the case of Mandelbrot set, it is known that $\M$ is locally connected at every point $c$ such that $z\mapsto z^2+c$ is not infinitely renormalizable. The first part of {\bf Question D} asks whether such a parameter lies in $\M_0$. If the answer to second part of {\bf Question D} is positive then $\M$ satisfies certain interesting properties; under those properties a planar continuum is locally connected if and only if it is path connected. The third part is about the estimation of $\lambda(\M)$ from above. Such a question  is largely motivated by fundamental results concerning the tricorn and other multicorns, stating that they are not path connected \cite{HS14,NS94,NS95}. A more recent result \cite{IM16} even points out that if $K$ is a multicorn there exist rational external rays that accumulate on a non-trivial arc belonging to the boundary of a hyperbolic component; therefore, we have $\lambda_K(x)\ge1$ for every limit point $x$ of such an external ray. Along this direction, the well known MLC may be modified to a more general question such as: \underline{\em Given  a multicorn $K$, is it true that $\lambda(K)\le n$ for some $n\ge1$?}

The rest of this paper is arranged as follows. Section \ref{A} gives the major steps to prove Theorem \ref{invariance} and summarizes the preliminary notions and results that will be needed. Section 3 proves Theorem \ref{closure}. Section \ref{C} proves Theorem \ref{embedding}.  And Sections 5 to 7 respectively prove Theorems \ref{connected-fiber} to \ref{final} listed in Section \ref{A}.  Section \ref{M} recalls known facts about the Mandelbrot set and proves Theorem \ref{E}.
Section \ref{ex} gives a couple of examples related to Theorem \ref{invariance}. Section \ref{scale-examples} gives some simple examples of planar compacta $K$ for which the function $\lambda_K$ can be explicitly determined.


\section{The Strategy to Prove Theorem A}\label{A}

This section outlines our proof for Theorem \ref{invariance}. We firstly note that the concept of Peano space has its origin in an ancient result by Sch\"onflies. See \cite[p.515, $\S61$, II, Theorem 10]{Kuratowski68}. Mostly due to this result, a compactum $K\subset\mathbb{C}$ is said to fulfill the \emph{Sch\"onflies condition} \cite{LLY-2017} provided that for any region $U$ bounded by two parallel lines $L_1$ and $L_2$, such that $\partial U=L_1\cup L_2$, the {\bf difference} $\overline{U}\setminus K$ has at most finitely many components intersecting both $L_1$ and $L_2$. Given $L_1$ and $L_2$, this happens if and only if $\overline{U}\cap K$ has at most finitely many components $Q_1,\ldots, Q_n$ intersecting both $L_1$ and $L_2$. 
See the following simplified depiction.
\begin{figure}[ht]
\vskip -0.25cm
\begin{center}
\begin{tikzpicture}[scale=0.618]
\draw  (-2,0)--(14,0);
\draw  (-2,5)--(14,5);

\draw[line width=2 pt ,color=gray](3,5) -- (2,3) -- (3.5,1.5) -- (3,0);
\draw[line width=2 pt ,color=gray] (6,5) -- (5,3) -- (6.5,1.5) -- (6,0);
\draw[line width=2 pt ,color=gray] (11,5) -- (10,3) -- (11.5,1.5) -- (11,0);

\node at (14.75,5) {$L_1$}; \node at (14.75,0) {$L_2$};
\node at (3.1,2.9) {$Q_1$};
\node at (6.1,2.9) {$Q_2$};
\node at (11.2,2.9) {$Q_n$};
\node at (9.0,1.4) {\Large$\cdots\cdots$};

\draw [fill=black] (11,0) circle (0.1);
\draw [fill=black] (11,5) circle (0.1);
\draw [fill=black] (6,0) circle (0.1);
\draw [fill=black] (6,5) circle (0.1);
\draw [fill=black] (3,0) circle (0.1);
\draw [fill=black] (3,5) circle (0.1);
\end{tikzpicture}
\end{center}
\vskip -1.0cm
\caption{The components $Q_1,\ldots, Q_n$ and the two lines $L_1,L_2$.}\label{long-band}
\end{figure}
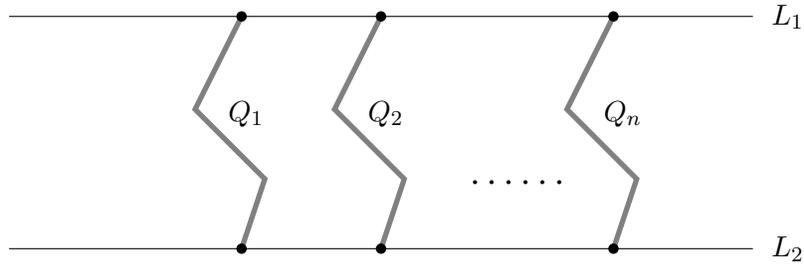

\begin{main-theorem}[cf.~{\cite[Theorem 1 and 3]{LLY-2017}}]\label{PS} A compactum $K\subset\mathbb{C}$ satisfies the Sch\"onflies condition if and only if it is a Peano space. In particular, every locally connected or finitely suslinian compactum $K\subset\mathbb{C}$ satisfies the Sch\"onflies condition and hence is a Peano space.
\end{main-theorem}

Theorem \ref{PS} provides an insight into Blokh-Curry-Oversteegen's locally connected and finitely suslinian models. Such an insight orients us to the natural extension to the Peano model $\Dc_K^{PS}$ of  a compactum $K\subset\bbC$, as developed in \cite{LLY-2017}. The corresponding decomposition $\Dc_K^{PS}$ is based on a symmetric relation $R_K$, which is recalled below.

\begin{defi}\label{R_K}
Given two disjoint simple closed curves $J_1$ and $J_2$, we denote by $U(J_1,J_2)$ the component of $\hat{\bbC}\setminus(J_1\cup J_2)$ that is bounded by $J_1\cup J_2$. This is an annulus  in $\hat{\bbC}$. Two points $x\ne y\in K$ are {\em related under $R_K$} if there exist two disjoint simple closed curves $J_1\ni x$ and $J_2\ni y$ such that $\overline{U(J_1,J_2)}\cap K$ contains an infinite sequence of components $P_n$ intersecting both $J_1$ and $J_2$, whose limit $P_\infty=\lim\limits_{n\rightarrow\infty}P_n$ under Hausdorff distance contains $\{x,y\}$. The following picture represents $\overline{U}$ as a closed annulus.
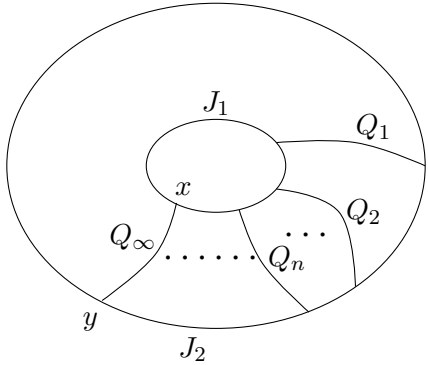
\begin{figure}[ht]
\begin{center}
\begin{tikzpicture}[scale=0.618]
\draw  (-0.5,3) ellipse (4.5 and 3.5);
\draw  (-0.5,3) ellipse (1.5 and 1);
\draw  plot[smooth, tension=.7] coordinates {(0.8,3.5) (2.5,3.5) (4,3)};
\node at (-0.55,1) {\Large$\cdots\cdots$};
\node at (1.5,1.5) {\Large$\cdots$};
\draw  plot[smooth, tension=.7] coordinates {(0.8,2.5) (2.15,2) (2.5,0.4)};
\draw  plot[smooth, tension=.7] coordinates {(-1.35,2.2) (-1.8,1.1) (-2.95,0.1)};
\draw  plot[smooth, tension=.7] coordinates {(0,2.05) (0.5,0.9) (1.5,-0.15)};
\node at (-0.5,4.35) {$J_1$};
\node at (-1,-0.95) {$J_2$};
\node at (2.8893,3.8286) {$Q_1$};
\node at (2.6786,2) {$Q_2$};
\node at (1.0286,1) {$Q_n$};
\node at (-2.2893,1.4679) {$Q_\infty$};
\node at (-1.2,2.5) {$x$};
\node at (-3.2,-0.35) {$y$};
\end{tikzpicture}
\end{center}\vskip -0.4cm
\caption{Relative locations of $P_k$ and $P_\infty$ inside the annulus $\overline{U(J_1,J_2)}$.}\label{R_K}
\end{figure}
\end{defi}

Following \cite{LLY-2017}, we also call $R_K$ the {\em Sch\"onflies relation on $K$} and denote the \emph{minimal closed equivalence containing} $R_K$ by $\sim$, or $\sim_K$ if it is necessary to emphasize the compactum $K$. This closed equivalence is  referred to as the \emph{Sch\"onflies equivalence on} $K$ \cite[Definition 4]{LLY-2017}.

By \cite[Proposition 5.1]{LLY-2017} every class $[x]$ of $\sim$ is a continuum, thus $\Dc_K=\{[x]: x\in K\}$ is a monotone decomposition. By \cite[Theorem 5 and 6]{LLY-2017},  the core decomposition $\Dc_K^{PS}$ equals  $\Dc_K$.

By \cite[Lemma 3.5]{LLY-2017}, the Sch\"onflies relation $R_K$ on a compact set $K\subset\mathbb{C}$ may be slightly generalized. See {\bf Theorem \ref{good-R}}. Such a generalization is useful in the proof for Theorem \ref{invariance} and is connected to some basic observations recalled in the next four paragraphs.

Suppose $W^*$ is an open region  whose boundary consists of $n\ge3$ disjoint simple closed curves $\Gamma_1, \ldots, \Gamma_n\subset\mathbb{C}$, with  $\Gamma_2,\ldots,\Gamma_n\subset Int(\Gamma_1)$. Here $Int(\Gamma_1)$ denotes the component of $\hat{\mathbb{C}}\setminus]\Gamma_1$ that does not contain $\infty$. Suppose that $\overline{W^*}\cap K$ has infinitely many components $P_k$ each of which intersects both $\Gamma_1$ and $\Gamma_2$. By \cite[Lemma 3.5]{LLY-2017} there exist $z_1\in\Gamma_1$ and $z_2\in\Gamma_2$ such that $(z_1,z_2)\in R_K$. More precisely, let $U$ be the annulus  with $\partial U=\Gamma_1\cup \Gamma_2$ and assume that $P_k\rightarrow P_\infty$ under Hausdorff distance.

If none of the curves $\Gamma_i$ with $i\ge3$ can be connected to $\Gamma_1\cup\Gamma_2$ by a simple arc $\alpha\subset (W^*\setminus K)$, then every region $Int(\Gamma_i)$ with $i\ge3$ is contained in a bounded component of $\mathbb{C}\setminus P$ for some component $P$ of $\overline{W^*}\cap K$. Let $K^*=\overline{Int(\Gamma_3)}\cup\cdots\cup\overline{Int(\Gamma_n)}$. Then, all but $n-2$ of the continua $P_k$ are disjoint from $K^*$ thus are components of $\overline{W^*}\cap (K\cup K^*)$. Since those $P_k$ are subsets of $\overline{U}\cap K$, each of them is a component of $\overline{U}\cap K$. Here $U$ is the open annulus with $\partial U=\Gamma_1\cup\Gamma_2$. Therefore, $(z_1,z_2)\in R_K$ for any $z_1\in(\Gamma_1\cap P_\infty)$ and $z_2\in(\Gamma_2\cap P_\infty)$.

Otherwise, a curve $\Gamma_i$ with $i\ge3$, say $\Gamma_n$, can be connected to $\Gamma_1\cup\Gamma_2$ by an arc $\alpha\subset (W^*\setminus K)$. With no loss of generality, we may assume that $\alpha$ has one end point on $\Gamma_2$, and the other on $\Gamma_n$. Then $\alpha$ may be thickened to a topological disk hose closure does not intersects $K$. Moreover, its interior is contained in $W^*\setminus K$ and its closure consists of two arcs $\alpha', \alpha''\subset (W^*\setminus K)$, together with one arc on $\Gamma_1$ and one on $\Gamma_2$. Using those arcs $\alpha', \alpha''$ we can find a new simple closed curve $\Gamma_2'\subset(\Gamma_2\cup\Gamma_n\cup\alpha'\cup\alpha'')$ and a new region $W$ whose boundary is $\Gamma_1\cup\Gamma_2'\cup\Gamma_3\cup\cdots\Gamma_{n-1}$.
\begin{figure}[ht]
\begin{center}
\begin{tikzpicture}[scale=.618]
\footnotesize
 \pgfmathsetmacro{\xone}{0}
 \pgfmathsetmacro{\xtwo}{21}
 \pgfmathsetmacro{\yone}{0}
 \pgfmathsetmacro{\ytwo}{10}

\foreach \p in {0.05,0.10,...,5.95}
    \draw[green,very thin] (\xtwo/7+\p,\ytwo/7) -- (\xtwo/7+\p,6*\ytwo/7);

\foreach \p in {0.05,0.10,...,5.95}
    \draw[green,very thin] (4*\xtwo/7+\p,\ytwo/7) -- (4*\xtwo/7+\p,6*\ytwo/7);

\foreach \p in {0.05,0.10,...,2.95}
    \draw[green,very thin] (3*\xtwo/7+\p,3.8*\ytwo/7) -- (3*\xtwo/7+\p,4.2*\ytwo/7);

\draw[gray,very thick] (\xone,\yone) -- (\xtwo,\yone) --  (\xtwo,\ytwo) --
 (\xone,\ytwo) --  (\xone,\yone);

\draw[gray,very thick] (\xone+\xtwo/7,\yone+\ytwo/7) -- (\xone+3*\xtwo/7,\yone+\ytwo/7) --  (\xone+3*\xtwo/7,\yone+6*\ytwo/7) --
(\xone+\xtwo/7,\yone+6*\ytwo/7) -- (\xone+\xtwo/7,\yone+\ytwo/7);

\draw[gray,very thick] (\xone+4*\xtwo/7,\yone+\ytwo/7) -- (\xone+6*\xtwo/7,\yone+\ytwo/7) --  (\xone+6*\xtwo/7,\yone+6*\ytwo/7) --
(\xone+4*\xtwo/7,\yone+6*\ytwo/7) -- (\xone+4*\xtwo/7,\yone+\ytwo/7);

\draw[blue,very thick] (\xone+4*\xtwo/7,\yone+4*\ytwo/7) -- (\xone+3*\xtwo/7,\yone+4*\ytwo/7);

\draw[red,very thick] (\xone+4*\xtwo/7,\yone+4.2*\ytwo/7) -- (\xone+3*\xtwo/7,\yone+4.2*\ytwo/7);

\draw[red,very thick] (\xone+4*\xtwo/7,\yone+3.8*\ytwo/7) -- (\xone+3*\xtwo/7,\yone+3.8*\ytwo/7);

 \draw[blue] (\xone,\ytwo) node[anchor=north west] {$\Gamma_1$};
 \draw[blue] (1.2*\xtwo/7,\ytwo/7) node[anchor=north east] {$\Gamma_2$};
 \draw[blue] (4.2*\xtwo/7,\ytwo/7) node[anchor=north east] {$\Gamma_n$};
 \draw[red] (3.5*\xtwo/7,3.85*\ytwo/7) node[anchor=north] {$\alpha'$};
 \draw[red] (3.5*\xtwo/7,4.2*\ytwo/7) node[anchor=south] {$\alpha''$};

\end{tikzpicture}
\end{center}\vskip -0.75cm
\caption{Relative location of $\Gamma_2'$ inside $W^*$.}\label{tunnel}
\end{figure}
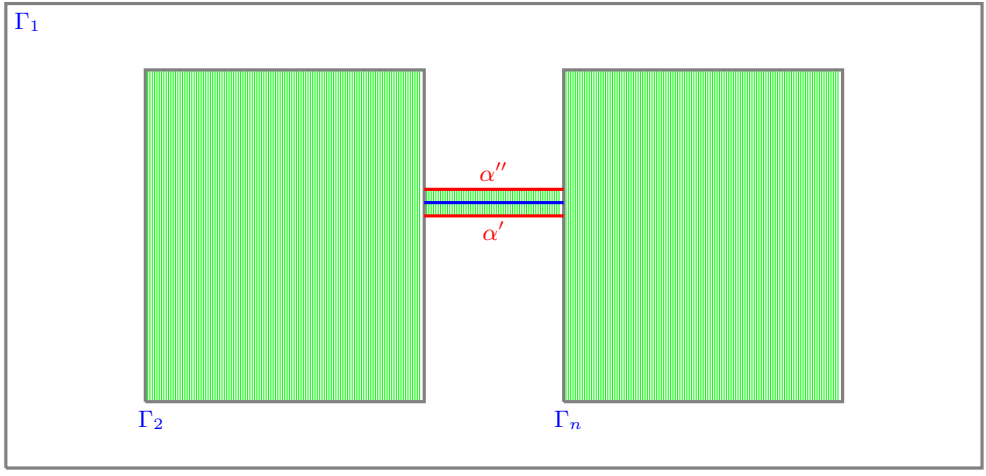
See Figure \ref{tunnel} for a simplified depiction of the region bounded by $\Gamma_2'$. From this it is immediate that all but finitely many of the above continua $P_k$ are also components of $\overline{W}\cap K$.

Repeating the same procedure on $W$ finitely many steps, if necessary, we can obtain two disjoint simple closed curves $\gamma_1$ and $\gamma_2$ such that all but finitely many of the above continua $P_k$ are also components of $\overline{U(\gamma_1,\gamma_2)}\cap K$. Here $U(\gamma_1,\gamma_2)$ is the annulus whose boundary equals $\gamma_1\cup\gamma_2$. From the construction of $\Gamma_2'$ we can infer that the curves $\gamma_1,\gamma_2$ satisfy $(\Gamma_i\cap K)\subset(\gamma_i\cap K)$ for $i=1,2$. Therefore, we can conclude that $(z_1,z_2)\in R_K$ for any $z_1\in(\Gamma_1\cap P_\infty)$ and $z_2\in(\Gamma_2\cap P_\infty)$.
Those arguments may be summarized as follows.

\begin{main-theorem}\label{good-R}
Given a compactum $K\subset\hat{\mathbb{C}}$ and two points $x,y\in K$. Then $(x,y)\in R_K$ if and only if there is an open region $W$ whose boundary consists of $n\ge2$ disjoint simple closed curves $\Gamma_1,\ldots,\Gamma_n$ such that $\overline{W}\cap K$ has infinitely many components $P_k$ intersecting both $\Gamma_1$ and $\Gamma_2$ whose limit $P_\infty=\lim_kP_k$ under Hausdorff distance satisfies $x\in(\Gamma_1\cap P_\infty)$ and $y\in(\Gamma_2\cap P_\infty)$.
\end{main-theorem}

Hereafter in this paper, let $\hat{\mathbb{C}}$ be equipped with the spherical distance $\rho$; let $D_r(x)$ denote the open ball on $\hat{\mathbb{C}}$ centered at $x$ with radius $r>0$.

The first step in proving Theorem \ref{invariance} is to analyze the structure of $R_K$ as a subset of $K\times K$, while the question whether $R_K$ is a closed relation remains open.  Actually, we will obtain a nontrivial characterization for the closure $\overline{R_K}$.

\begin{main-theorem}\label{closure}
Two points $x\ne y\in K$ are related under $\overline{R_K}$ if and only if $K\setminus(D_r(x)\cup D_r(y))$ has infinitely many components intersecting both $\partial D_r(x)$ and $\partial D_r(y)$, for $0<r<\frac{\rho(x,y)}{2}$.
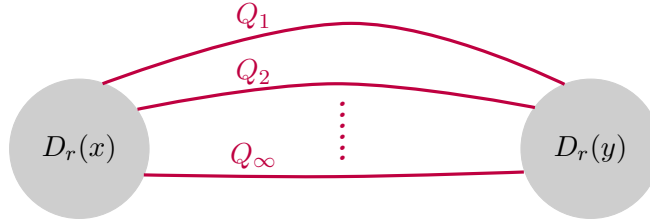
\begin{center}
\begin{tikzpicture}[scale=0.618]
\draw[fill,gray!38.2]  (-7.5,1.5) circle (1.5);
\draw[fill,gray!38.2]  (3.5,1.5) circle (1.5);

\draw[very thick, purple]  plot[smooth, tension=.7] coordinates {(-7,2.9) (-1.65,4.2) (2.9375,2.8875)};
\draw[very thick, purple]  plot[smooth, tension=.7] coordinates {(-6.25,2.35) (-2,2.9) (2.3,2.3875)};
\draw[very thick, purple]  plot[smooth, tension=.7] coordinates {(-6.1125,0.9375) (-2,0.9) (2.0625,1)};

\node[very thick, purple] at (-3.75,4.35) {$Q_1$};
\node[very thick, purple] at (-3.75,3.15) {$Q_2$};
\node[very thick, purple] at (-3.75,1.35) {$Q_\infty$};

\node at (-7.5,1.5) {$D_r(x)$};
\node at (3.5,1.5) {$D_r(y)$};

\node[very thick, purple] at (-1.85,2.4) {\LARGE$\vdots$};
\node[very thick, purple] at (-1.85,1.7) {\LARGE$\vdots$};

\end{tikzpicture}
\end{center}
\end{main-theorem}

The second step is to verify that the fibers of $\overline{R_K}$ are each a continuum.

\begin{main-theorem}\label{connected-fiber}
For any $x\in K$ the fiber  $\overline{R_K}[x]=\left\{y: (x,y)\in\overline{R_K}\right\}$ at $x$ is connected.
\end{main-theorem}

The third step is to show that every fiber of $\overline{R_{f^{-1}(K)}}$ is mapped by $f$ onto a fiber of $\overline{R_K}$. This step uses a well known result by Beardon \cite[p.95, Lemma 5.7.2]{Beardon91}.

\begin{main-theorem}\label{invariant-fiber}
If $f(u)=x$ for $u\in L:=f^{-1}(K)$ and $x\in K$ then $f\left(\overline{R_{L}}[u]\right)=\overline{R_K}[x]$.
\end{main-theorem}

It is routine to verify that if $f\left(\overline{R_{L}}[u]\right)\subset\overline{R_K}[x]$ for all $u$ with $f(u)=x$ then the collection $f^{-1}\left(\Dc_K^{PS}\right)$ defined in Theorem \ref{invariance} is a monotone decomposition that is refined by the core decomposition $\Dc_{L}^{PS}$. See Corollary \ref{forward-invariance}. From this we can set $\tilde{f}(E)$ for $E\in\mathcal{D}_K^{PS}$ to be the unique element of $\Dc_K^{PS}$ that contains $f(E)$ and define a continuous map $\tilde{f}:\Dc_{L}^{PS}\rightarrow\Dc_K^{PS}$ that satisfies $\pi_K\circ f=\tilde{f}\circ\pi_{L}$. Here $\pi_K: K\rightarrow\Dc_K^{PS}$ and $\pi_{L}: L\rightarrow\Dc_{L}^{PS}$ denote the two natural projections from $K$ and $L$   onto the core decompositions, respectively. In other words, we will find a factor system for the restriction $\left.f\right|_L: L\rightarrow K$, which is of particular interest when $K$ is the Julia set of $f$. In such a case, we have $L=K$ and the following commutative diagram:
\[
\begin{array}{ccc}K&\xrightarrow{\hspace{0.6cm}f\hspace{0.6cm}}&K\\ \downarrow\!\pi&&\downarrow\!\pi\\ \Dc_K^{PS}&\xrightarrow{\hspace{0.6cm}\tilde{f}\hspace{0.6cm}}&\Dc_K^{PS}\end{array}\]

In the fourth step we continue  to introduce an equivalence $\approx$ on $K$ by requiring that $x\approx y$ if and only if the union of the elements of $\Dc_{L}^{PS}$ intersecting $f^{-1}(x)$ equals the union of those that intersect $f^{-1}(y)$. Namely, we have
\begin{equation}\label{new-equiv}
x\approx y\quad\Leftrightarrow\quad \pi_{L}\left(f^{-1}(x)\right) = \pi_{L}\left(f^{-1}(y)\right).
\end{equation}
Here one may use the openness of $f$ to show that $\approx$ is a closed equivalence and then apply Beardon's result \cite[p.95, Lemma 5.7.2]{Beardon91} to the fibers of $\overline{R_K}$ and to further verify that  $\approx$ contains $\overline{R_K}$ hence contains $\sim$, the smallest closed equivalence on $K$ containing $\overline{R_K}$. Moreover, from Theorem \ref{invariant-fiber} and the definition of $\approx$  we can infer that for any $x\in K$ and any $u\in f^{-1}(x)$ the element $\pi_L(u)\in\Dc_L^{PS}$ is sent by $f$ onto the class of $\approx$ that contains $x$. See Section \ref{coincidence} for further details. Moreover, we even have.

\begin{main-theorem}\label{final}
The two equivalences $\approx$ and $\sim$ are equal.
\end{main-theorem}

Finally, based on Theorems \ref{invariant-fiber} to \ref{final} and the fundamental properties of $\approx$ mentioned as above, we may summarize the proof for  Theorem \ref{invariance} as follows.

\begin{proof}[\bf Proof for Theorem A]
If $E\in\Dc_K^{PS}$ and $N$ is a component of $f^{-1}(E)$, then for any $x\in E$ and any $u\in N$ the pair $(x,f(u))$ belongs to $\sim$, which is equal to $\approx$ by Theorem \ref{final}. By Equation \ref{new-equiv} there exists some $v\in f^{-1}(x)$ with $\pi_L(u)=\pi_L(v)$. By Theorem \ref{invariant-fiber} the decomposition $f^{-1}\left(\Dc_K^{PS}\right)$ is refined by $\Dc_{L}^{PS}$. Therefore, we have $\pi_L(u)\subset N$. Since $\displaystyle \left\{\pi_L(w): w\in f^{-1}(x)\right\}$ is a finite family of disjoint continua, the flexibility of $u\in N$ requires that $N=\pi_L(v)=\pi_L(u)$. In other words, the component $N$ of $f^{-1}(E)$ is also an element of $\Dc_{L}^{PS}$, which is exactly what we want to prove.
\end{proof}

\section{The closure  $\overline{R_K}$ of $R_K$ as a subset of $K\times K$}

In this section we will characterize the closure of $R_K$ as a subset of $K\times K$. The whole section is a complete proof for Theorem \ref{closure}, in which a couple of fundamental steps appear as lemmas. To start off, we prepare some basic issues.


Firstly, we note that the ``if'' part is clear and we just discuss the ``only if'' part. To this end, we consider any two disjoint simple closed curves $J_1\ni x$ and $J_2\ni y$ such that $K\cap\overline{U(J_1,J_2)}$ contains an infinite sequence of distinct components $P_n$ intersecting both $J_1$ and $J_2$, whose limit $P_\infty=\lim\limits_{n\rightarrow\infty}P_n$ under Hausdorff distance contains $\{x,y\}$. Let $P_0$ be the component of $K\cap\overline{U(J_1,J_2)}$ that contains $P_\infty$. We may assume that $P_0\cap P_n=\emptyset$ for all $n\ge1$. It suffices to verify that for $0<r<\frac12{\rm dist}(J_1,J_2)$ the following holds:
\begin{center}\com{$K\setminus(D_r(x)\cup D_r(y))$ has infinitely many components intersecting both $\partial D_r(x)$ and $\partial D_r(y)$.}
\end{center}
Here, $U(J_1,J_2)$ denotes the component of $\hat{\mathbb{C}}\setminus(J_1\cup J_2)$ bounded by $J_1\cup J_2$, which is topologically an open annulus. By Sch\"onflies Theorem \cite[pp.71-72, Theorem
3 and 4]{Moise}, we may apply an appropriate homeomorphism $\phi:\hat{\mathbb{C}}\rightarrow\hat{\mathbb{C}}$ that sends  $U(J_1,J_2)$ onto $\{z:1<|z|<2\}$.  Clearly, $(x,y)\in R_K$ if and only if $(\phi(x),\phi(y))\in R_{\phi(K)}$. Therefore, we may assume in the rest of this section that $J_1=\{z:|z|=1\}$ and $J_2=\{z:|z|=2\}$.

Now, we consider the compact set $K\cap\overline{U(J_1,J_2)}$. Recall that $P_1$ is disjoint from $P_0\cup\left(\bigcup_{j\ne 1}P_j\right)$, which is a compact set. Clearly, no component of $K\cap\overline{U(J_1,J_2)}$ intersects $P_1$ and $P_0\cup\left(\bigcup_{j\ne 1}P_j\right)$ both. Since all components of $K\cap\overline{U(J_1,J_2)}$ (including all those $P_n$) are each a quasi-component of $K\cap\overline{U(J_1,J_2)}$ (see for instance \cite[p.169, \S 47,II, Theorem 2]{Kuratowski68}), we can use  compactness of $P_0\cup\left(\bigcup_{j\ne 1}P_j\right)$ to find a separation $K\cap\overline{U(J_1,J_2)}=E\cup F$ satisfying
\begin{equation}
 P_1\subset E\quad \text{and}\quad
 \left(P_0\cup\left(\bigcup_{j\ne 1}P_j\right)\right)\subset F.
\end{equation}
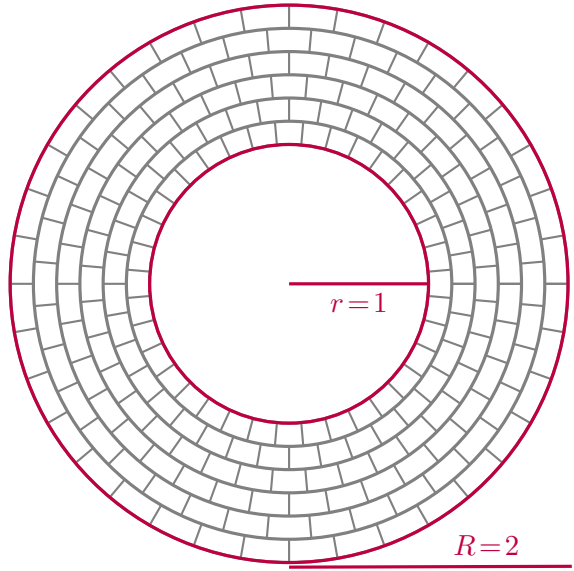
\begin{figure}[ht]
\centering\vskip -0.25cm
\begin{tikzpicture}[scale=0.618]
\draw[purple, very thick] (0:0) -- (0:3);  \draw[purple] (0:1.5) node[below] {$r\!=\!1$};
\draw[purple, very thick] (270:6.1) -- (315:6.1*1.41);  \draw[purple] (305:7.4) node[above] {$R\!=\!2$};

\foreach \r in {3,3.5,...,6.0}        \draw[gray,very thick] (0,0) circle (\r);
\foreach \r in {3,4,5}    \foreach \b in {5,15,...,355}
\draw[gray,thick] (\b:\r) -- (\b:0.5+\r);

\foreach \r in {3.5,4.5,5.5}    \foreach \b in {0,10,...,360}
\draw[gray,thick] (\b:\r) -- (\b:0.5+\r);

\foreach \r in {3,6.0}        \draw[purple,very thick] (0,0) circle (\r);
\end{tikzpicture}\vskip -0.5cm
\caption{Polar Brick Wall Tiling}\label{brick-wall}
\end{figure}
Cover the annulus $U(J_1,J_2)$ by a polar brick wall tiling $\mathcal{T}$, as indicated in Figure \ref{brick-wall}, whose tiles are all of diameter smaller than $\frac{1}{2}{\rm dist}(E, F)$.
The union of all the tiles $T\in\mathcal{T}$ with $T\cap E\ne\emptyset$ has finitely many components. Denote by $P_1^*$ the component that contains $P_1$.
Let $W_1$ be the component of $\hat{\mathbb{C}}\setminus P_1^*$ containing $\infty$. Then $W_1$ contains $P_\infty$ (hence $P_0$) and every $P_j$ with $j\neq 1$. Since $P_1^*$ is a continuum with no cut point, by Torhorst Theorem \cite[p.126]{Whyburn79}, the boundary of $W_1$ is a simple closed curve $\Gamma_1$. Moreover, $\Gamma_1$ contains exactly two sub-arcs $\alpha_1, \beta_1$ each of which intersects both $J_1$ and $J_2$ at the end points and is otherwise contained in the open annulus $U(J_1,J_2)$. Clearly, $\Theta=J_1\cup J_2\cup\alpha_1\cup\beta_1$ cuts the sphere into four Jordan domains, each of which is bounded by a simple closed curve. Let $D_1$ be the component of $\hat{\mathbb{C}}\setminus \Theta$ that intersects $P_1$. Let $D_\infty$ be the component that intersects $P_\infty$. Then $D_1\cap D_\infty=\emptyset$.

Let $I_i=\overline{D_\infty}\cap J_i$ for $i=1,2$. Then $I_i$ is a sub-arc of $J_i$ and $\partial D_\infty=I_1\cup I_2\cup\alpha_1\cup\beta_1$; moreover, we also have $\overline{D_1}\cap\overline{D_\infty}=\alpha_1\cup \beta_1$.  Since $P_\infty\subset \overline{D_\infty}$ and since $P_0\cap(\alpha_1\cup\beta_1)=\emptyset$, the convergence $P_n\rightarrow P_\infty$ under Hausdorff distance implies that $P_n\subset \overline{D_\infty}$ for all but finitely many $n>1$. We may assume that $P_n\subset \overline{D_\infty}$ for all $n\ge2$. Clearly, all those $P_n$ with $n\ge2$ are also components of $K\cap\overline{D_\infty}$, since each of them is a component of $K\cap\overline{U(J_1,J_2)}$ and since $\overline{D_\infty}\subset\overline{U(J_1,J_2)}$.

Denote the two components of $\overline{D_\infty}\setminus P_0$ containing $\alpha_1$ and $\beta_1$ by $U_L$ and $U_R$, respectively. Then either $U_L$ or $U_R$ contains infinitely many of the components $P_n$. Without loss of generality, we may assume that all $P_n (n\geq 2)$ are contained in $U_R$. Now we fix a small enough number $\delta\in(0,r)$ and consider the two disks $B_\delta(x)=\{z: |x-z|<\delta\}$ and $B_\delta(y)=\{z: |y-z|<\delta\}$.  The convergence $P_n\rightarrow P_\infty$ under Hausdorff distance indicates that there exist infinitely many $P_n\subset U_R$ intersecting $B_\delta(x)$ and $B_\delta(y)$ at the same time. By going to an appropriate sub-sequence, we may assume that all $P_n (n\geq 2)$ intersect both $B_\delta(x)$ and $B_\delta(y)$.

Recall that $\overline{D_\infty}$ is a topological disk. Without changing its topology, we may represent $\overline{D_\infty}$ as $[0,1]^2$. Moreover, for $i=1,2$ let $I_i'$ be the irreducible sub-arc of $I_i$ that intersects $\alpha_1$ and $P_2$ at the same time. The following Figure \ref{locations}
\begin{figure}[ht]
\vskip -0.4cm
\centering
\begin{tikzpicture}[scale=0.99]
\draw  (0,0) rectangle (11,7);
\coordinate [label=135:$x$] (x) at (4.5,7);
\coordinate [label=210:$y$] (y) at (3.5,0);
\draw[line width=2.5pt,color=gray] (5,7)--(x) -- (3,5) --(4.5,2.5)-- (y)--(4,0);

\draw[line width=2.5pt,color=gray] (8,7)-- (7.5,5.8)--(5.5,5.5)--(7.5,4) --(7.5,2)--(5,1)--(6.5,0.5)--(6.5,0);
\draw[line width=2pt,red](7.5,7) node (v1) {}--(5,7);
\draw[line width=2pt,red](6,0)--(4,0);

\draw[line width=2pt,color=blue] (8,7) -- (11,7);
\draw[line width=2pt,color=blue] (6.5,0) -- (11,0);

\draw  (x)[fill=black]circle(0.06);
\draw  (y)[fill=black]circle(0.06);

\draw[line width=1pt,color=gray] (2.5,7) arc(180:360:2cm);
\draw[line width=1pt,color=gray] (1.5,0) arc(180:0:2cm);
\draw[line width=1.25pt,color=blue] (7.5,7)-- (7.5,6.5)--(7,6.5)--(7,6) --(5,6)--(5,5.05)--(5.5,5.05)--(5.5,4.5)--(6.5,4.5)--(6.5,4)--(7,4)--(7,2)--(5.5,2)--(5.5,1.5)--(4.5,1.5)--(4.5,0.5)--(6,0.5)--(6,0);

\draw  (7.5,7)[fill=black]circle(0.06)node[above]{$a$};
\draw  (9.5,7)node[below]{$I_1'$};
\draw  (0,3.5)node[right]{$\beta_1$};   \draw  (11,3.5)node[left]{$\alpha_1$};
\draw  (5,5.05)[fill=black]circle(0.06)node[below]{$c$};
\draw  (4.82,1.5)[fill=black]circle(0.06)node[above]{$d$};
\draw  (6,0)[fill=black]circle(0.06)node[below]{$b$};
\draw  (8.75,0)node[above]{$I_2'$};

\draw  (5,7)[fill=black]circle(0.06)node[above]{$x'$};
\draw  (3.35,5.35)[fill=black]circle(0.06)node[below]{$c'$};
\draw  (4.30,1.82)[fill=black]circle(0.06);
\node at (3.9,1.68){$d'$};
\draw  (4,0)[fill=black]circle(0.06)node[below]{$y'$};

\node at (6.6,3.5) {$\alpha_2$};
\node at (7.9,3.5) {$P_2$};
\node at (4.3,3.63) {$P_\infty$};
\end{tikzpicture}
\vskip -0.75cm
\caption{A simplified depiction for $D_\infty$ with helpful markings.}\label{locations}
\end{figure}
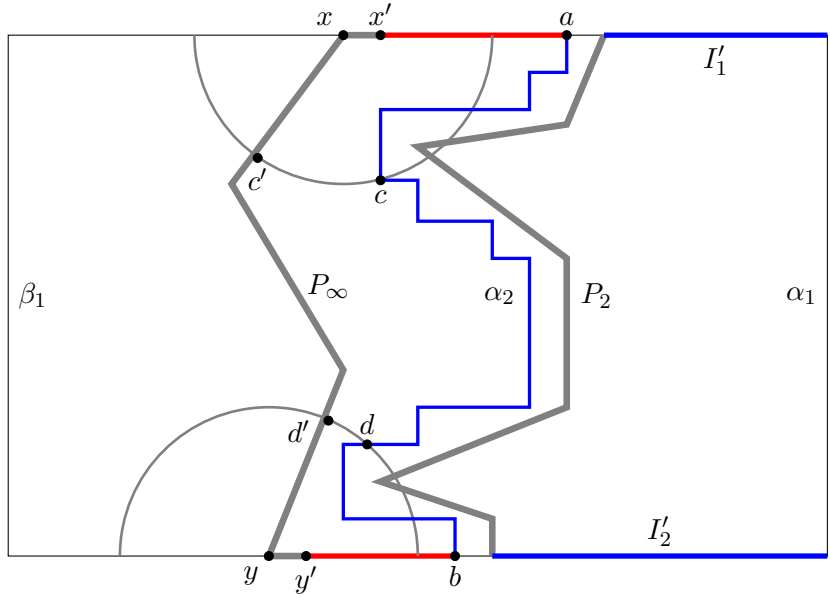
is a simplified depiction for relative locations of the arcs $\alpha_1,\beta_1,I_1,I_2$, the component $P_2$, and the two semi-circles $\overline{D_\infty}\cap\partial B_\delta(x)$ and $\overline{D_\infty}\cap\partial B_\delta(y)$. Let $A$ be the union of $P_2\cup I_1'\cup I_2'$ with all the ``bounded'' components of $\hat{\mathbb{C}}\setminus(P_2\cup I_1'\cup I_2')$, that do not contain $\infty$. The all but finitely many of $P_n$ with $n\ge3$ are disjoint from $A$. Assume that $P_n\cap A=\emptyset$ for $n\ge3$. Then $A$ and $B=P_0\cup\left(\bigcup_{n\ge3}P_n\right)$ are disjoint compact subsets of $(K\cup A)\cap \overline{D_\infty}$ and the following lemma may be applied, in which we choose $u\in A$  and $v\in P_0$.

\begin{lemma}[cf.~{\cite[p.108, (3.1) Separation Theorem]{Whyburn42}}]\label{separation-theorem}
If $A,B\subset\mathbb{R}^2$ are compact sets such that $A\cap B$ is totally disconnected and $u,v$ are points of $A\setminus B$ and $B\setminus A$, respectively, and $\epsilon$ is any positive number, then there esists a simple closed curve $J$ which separates $u$ and $v$ and is such that $J\cap(A\cup B)$ is contained in $A\cap B$ and that every point of $J$ is at a distance less than $\epsilon$ from some point of $A$.
\end{lemma}

If we choose the number $\epsilon$ in Lemma \ref{separation-theorem} to be smaller than $\frac12{\rm dist}(A,B)$ then we may find a simple arc $\alpha_2\subset(J\cap\overline{D_\infty})$ starting from a point $a\in I_1$ and ending at a point $b\in I_2$ such that $D_\infty\setminus\alpha_2$ is the union of two Jordan domains, one contains $A$ and the other $B$. Check Figure \ref{locations}, in which the arc $\alpha_2$ is represented as a broken line. Since $P_2$ intersects each of the two semi-circles $\overline{D_\infty}\cap\partial B_\delta(x)$ and $\overline{D_\infty}\cap\partial B_\delta(y)$, so does $\alpha_2$.

Let $x'$ be the point of $P_\infty\cap I_1$ that is closest to $a$, and $y'$ the point of $P_\infty\cap I_2$ closest to $b$.
Let $c$ be the last point of $\alpha_2$ at which $\alpha_2$ leaves $\partial B_\delta(x)$ and $d$ the first point after $c$ at which $\alpha_2$ meets $\partial B_\delta(y)$. The sub-arc of $\alpha_2$ from $c$ to $d$ is denoted by $\alpha_2'$.
Let $c'$ be the point of $\partial B_\delta(x)\cap P_0$ that is closest to $c$. Let $d'$ be the point of $\partial B_\delta(y)\cap P_0$ that is closest to $d$.

Let $M=P_\infty\cup\overline{ax'}\cup\alpha_2\cup\overline{by'}$. Then $M$ is a continuum and exactly one of the components of $\hat{\mathbb{C}}\setminus M$, denoted $V$, has a boundary that contains $\overline{ax'}\cup\alpha_2\cup\overline{by'}$.
Let $N=P_\infty\cup\widetilde{cc'}\cup\alpha_2'\cup\widetilde{dd'}$. Then  $N$ is a continuum and exactly one of the components of $\hat{\mathbb{C}}\setminus N$, denoted $W$, has a boundary that contains $\widetilde{cc'}\cup\alpha_2'\cup\widetilde{dd'}$.
Clearly, we have $W\subset V\subset U(J_1,J_2)$.
Since $V\subset U(J_1,J_2)$ we see that every $P_n$ with $n\geq 3$ is a component of $K\cap\overline{V}$. The last step of our proof is to obtain
\begin{lemma}\label{key-step}
For all $n\ge3$ there is a component $Q_n$ of $K\cap\overline{W}$ with $Q_n\subset P_n$ that intersects $\partial B_\delta(x)$ and $\partial B_\delta(y)$ both. Moreover, all those $Q_n$ are each a component of $K\setminus(B_\delta(x)\cup B_\delta(y))$.
\end{lemma}

Applying \cite[p.73, Boundary Bumping Theorem I]{Nadler92} to each of the components $P_n(n\ge3)$, we can infer that every component of $P_n\cap\overline{W}$ intersects either $\partial B(x,\delta)$ or $\partial B(y,\delta)$. This implies that one of those components, denoted as $Q_n$, intersects both $\partial B_\delta(x)$ and $\partial B_\delta(y)$. Indeed, the compact set $P_n\cap\overline{W}$ is the union of two compact subsets, one is formed by the components of $P_n\cap\overline{W}$ intersecting  $\partial B_\delta(x)$ and the other by those intersecting  $\partial B_\delta(y)$. Denote these two sets as $P_{n,x}$ and $P_{n,y}$, respectively. Then $P_{n,x}\cup(P_n\cap B_\delta(x))$ and $P_{n,y}\cup(P_n\cap B_\delta(y))$ are two compact sets whose union is exactly $P_n$. Therefore, connectedness of $P_n$ ensures that $P_{n,x}\cap P_{n,y}\ne\emptyset$.

We claim that $Q_n$ ( as a component of $P_n\cap\overline{W}$ ) is also a component of $K\cap\overline{W}$, which verifies the former part of Lemma \ref{key-step}.  Actually, let $Q_n'$ be the component of $K\cap\overline{W}$ containing $Q_n$, then we have $Q_n'\subset P_n$ and $Q_n'\subset\overline{W}$. This indicates that $Q_n'$ is a connected subset of $P_n\cap\overline{W}$. Therefore, $Q_n'$ is contained in $Q_n$, the component of  $P_n\cap\overline{W}$ that intersects $Q_n'$.

Before verifying the latter part of Lemma \ref{key-step}, we recall that all the quasi-components of $K\cap\overline{W}$ are connected and coincide with  the components. See~\cite[p.169, \S 47, II, Theorem 2]{Kuratowski68}. Using this result, we may find for each $n\geq3$ a separation $K\cap\overline{W}=A_n\cup B_n$ with $Q_n\subset A_n$ and $P_0\cap A_n=\emptyset$. Here $A_n$ and $B_n$ are separated in the sense that $\overline{A_n}\cap B_n=\emptyset=A_n\cap\overline{B_n}$. Since $Q_n$ is a component of $A_n$ for all $n\ge3$, we only need to check that
\[K\setminus(B_\delta(x)\cup B_\delta(y))=A_n\cup B_n\cup(K\setminus M)\]
is a separation, where $M=B_\delta(x)\cup B_\delta(y)\cup \overline{W}$. To this end, we firstly  infer that the open arcs $\widetilde{cc'}$ and $\widetilde{dd'}$ are in the interior $M^o$ of $M$. Actually, there is an open disk $B_\epsilon(p)$  for any $p\in\widetilde{cc'}$,
satisfying $\overline{B_\epsilon(p)}\cap\left(P_0\cup\widetilde{dd'}\cup\alpha_2\right)=\emptyset$ and $B_r(p)\cap\partial B_\delta(x)\subset\widetilde{cc'}$. Clearly, the arc $\widetilde{cc'}$ divides the ball $B_r(p)$ into two parts, the part lying in $B_\delta(x)$ is denoted as $V_1$ and the other part $V_2$. As $V_2\cap\partial W=\emptyset=\overline{V_1}\cap W$ and $p\in \partial V_2$, we know that $V_2$ intersects $W$ and is contained in $W$. Therefore, we have $\widetilde{cc'}\subset M^o$. Since a similar argument ensures $\widetilde{dd'}\subset M^o$, we have $\overline{K\setminus M}\cap\left(\widetilde{cc'}\cup\widetilde{dd'}\right)=\emptyset$.
Combining this with the containment $A_n\subset \left(W\cup\widetilde{cc'}\cup\widetilde{dd'}\right)$, we see that $A_n$ and $K\setminus C$ are separated, {\em i.e.} $(K\setminus C)\cap \overline{A_n}=\emptyset=\overline{K\setminus C}\cap A_n$. Since $A_n$ and $B_n$ are  separated, we have verified that $K\setminus(B_\delta(x)\cup B_\delta(y))=A_n\cup\left(B_n\cup (K\setminus C)\right)$ is a separation.  \qed


\section{Peano model is independent of the embedding of $K$ into $\hat{\mathbb{C}}$}\label{C}

The proof for Theorem \ref{embedding} is related to a new set-function very similar to FitzGerald-Swingle's T-function \cite{FitzGerald67}, which plays a crucial role in the study of semi-locally connected model of continua.


\begin{deff}\label{S-function}
Given a compactum $K$ and a point $x\in K$, let $S(x)$ consist of all the points $y\in K$ satisfying the following property: there do not exist two disjoint open sets $U_x, U_y \subset K$ with $x\in U_x$ and $y\in U_y$ such that $K\setminus\left(U_x\cup U_y\right)$ has at most finitely many components intersecting $\partial U_x$ {and} $\partial U_y$ at the same time.
\end{deff}

\begin{lemma}\label{shrinking}
Given two disjoint open sets $U_1, U_2$ in a compactum $K$ such that $K\setminus\left(U_1\cup U_2\right)$ has at most finitely many components intersecting $\partial U_1$ {and} $\partial U_2$ at the same time. Denote those components as $Q_1,\ldots,Q_n$. Then,  for any open sets $W_1\subset U_1$ and $W_2\subset U_2$, the difference $K\setminus\left(W_1\cup W_2\right)$ has at most $n$ components intersecting both $\partial W_1$ {and} $\partial W_2$.
\end{lemma}
\begin{proof}
Actually, denote by $\mathcal{P}_W$ the family of all the components of $K\setminus\left(W_1\cup W_2\right)$ that intersect $\partial W_1$ and $\partial W_2$ at the same time. Then, applying the well known Boundary Bumping Theorem II \cite[Theorem 5.6, p74]{Nadler92} to each component $P\in\mathcal{P}_W$, we can infer that every component of  $P\setminus\left(U_1\cup U_2\right)$ intersects either $\partial U_1$ or $\partial U_2$. Let $A_1$ be the union of $P\cap U_1$ with all the components of  $P\setminus\left(U_1\cup U_2\right)$ intersecting $\partial U_1$  and $A_2$ be the union of $P\cap U_2$ with all those that intersect $\partial U_2$. The connectedness of $P=A_1\cup A_2$ then implies that $P\setminus\left(U_1\cup U_2\right)$ has at least one component, denoted as $P'$, that intersects both $\partial U_1$ and $\partial U_2$. Since $P'$ is also component of $K\setminus\left(U_1\cup U_2\right)$, we have $P'\in\{Q_1,\ldots,Q_n\}$. Clearly, for $P_1\ne P_2\in\mathcal{P}_W$ we necessarily have $P_1'\ne P_2'$. Therefore, the collection $\mathcal{P}_W$ has at most $n$ members.
\end{proof}


If $K$ in the above Definition \ref{S-function} is a subset of $\hat{\mathbb{C}}$ then the S-function is closely connected with the Sch\"onflies relation $R_K$. Actually, for any $x,y\in K$ the result of Theorem \ref{closure} ensures that $(x,y)\in\overline{R_K}$  if and only if there does not exist a positive number $r<\frac{\rho(x,y)}{2}$ such that $K\setminus(D_r(x)\cup D_r(y))$ has at most finitely many components intersecting both $\partial D_r(x)$ and $\partial D_r(y)$. By Lemma \ref{shrinking}, we can verify that this happens if and only if $y\in S(x)$. Therefore, we have the following characterization of $S(x)$ via the
Sch\"onflies relation $R_K$.

\begin{theo}\label{S-basic}
If $K\subset\hat{\mathbb{C}}$ is a compactum then $S(x)=\overline{R_K}[x]$ for all  $x\in K$.
\end{theo}

The following theorem relates the S-function on $K$ to that on a homeomorphic image of $K$.

\begin{theo}\label{S-function invariant}
If $h: K\rightarrow L\subset\hat{\mathbb{C}}$ is an embedding then $h(S(x))=S(h(x))$ for any $x\in K$.
\end{theo}
\begin{proof}
We only need to verify the inclusion $h(y)\in S(h(x))$ for any $y\in K$ with $y\in S(x)$. To this end, we fix two disjoint open sets $V_x, V_y\subset L$ with $h(x)\in V_x$ and $h(y)\in V_y$. Choose two open sets $U_x, U_y\subset K$ with $x\in U_x\subset h^{-1}(V_x)$ and $y\in U_y\subset h^{-1}(V_y)$. Since $y\in S(x)$ the difference $K\setminus\left(U_x\cup U_y\right)$ has infinitely many components intersecting both $\partial U_x$ {and} $\partial U_y$, denoted as $\{P_n\}$.
Then $L\setminus h(U_x\cup U_y)$ has infinitely many components $\{h(P_n)\}$ intersecting $h(\partial U_x)$ {and} $h(\partial U_y)$ at the same time. By Boundary Bumping Theorem II \cite[Theorem 5.6, p74]{Nadler92} 
, every component of  $h(P_n)\setminus\left(V_x\cup V_y)\right)$ intersects either $\partial V_x$ or $\partial V_y$. Let $A_n$ be the union of $h(P_n)\cap V_x$ with all the components of  $h(P_n)\setminus\left(V_x\cup V_y)\right)$ intersecting $\partial V_x$  and $B_n$ be the union of $h(P_n)\cap V_y$ with all those that intersect $\partial V_y$. Then $h(P_n)=A_n\cup B_n$ and the connectedness of $h(P_n)$ implies that $h(P_n)\setminus\left(V_x\cup V_y\right)$ has at least one component, denoted as $Q_n$, that intersects both $\partial V_x$ and $\partial V_y$. Since every $Q_n$ is also a component of $L\setminus\left(V_x\cup V_y\right)$, the flexibility in choosing $V_x$ and $V_y$ indicates that $h(y)\in S(h(x))$.
\end{proof}

Now we have all the ingredients to prove Theorem \ref{embedding}.

\begin{proof}[\bf Proof for Theorem C]
Given a compactum $K\subset\hat{\mathbb{C}}$ and an arbitrary embedding $h: K\rightarrow\hat{\mathbb{C}}$. Theorem \ref{S-function invariant} implies that $h(y)\in S(h(x))$ for any $x,y\in K$ with $y\in S(x)$. Since  $y\in S(x)$ if and only if $(x,y)\in\overline{R_K}$ and since  $h(y)\in S(h(x))$ if and only if $(h(x),h(y))\in\overline{R_{h(K)}}$, the product map $h\times h(x,y)=(h(x),h(y))$ is a homeomorphism between $K\times K$ and $h(K)\times h(K)$ that sends $\overline{R_K}$ onto $\overline{R_{h(K)}}$. By \cite[Theorem 7]{LLY-2017}, the core decomposition $\Dc_K^{PS}$ is given by the smallest closed equivalence containing $\overline{R_K}$. Therefore $\left\{h(d): d\in\Dc_K^{PS}\right\}$ is a monotone decomposition and is refined by $\Dc_{h(K)}^{PS}$. By symmetry $\left\{h^{-1}(e): e\in\Dc_{h(K)}^{PS}\right\}$ is a monotone decomposition and is refined by $\Dc_{K}^{PS}$. Combining these we have  $\left\{h(d): d\in\Dc_K^{PS}\right\}=\Dc_{h(K)}^{PS}$. This ends our proof.
\end{proof}

\section{Fibers of $\overline{R_K}$ are connected}

This section discusses the fibers $\overline{R_K}[x]=\{z: (x,z)\in \overline{R_K}\}$ of $\overline{R_K}$ and proves Theorem \ref{connected-fiber}.

\begin{proof}[{\bf Proof for Theorem \ref{connected-fiber}}]
Suppose that two points $x$ and $y$ on $K$ are related under the relation $\overline{R_K}$, we will show that there exists a subcontinuum of $\overline{R_K}[x]$ which contains $x$ and $y$.

For all $n\ge1$ with $\frac{1}{n}\le\frac{|x-y|}{3}$, the compact set $K\setminus(B_{1/n}(x)\cup B_{1/n}(y))$ has infinitely many components with each intersecting both $\partial B_{1/n}(x)$ and $\partial B_{1/n}(y)$. Picking an infinite subsequence of those components that converge to a limit under Hausdorff distance. Denote this limit as $P_n$. By the definition of $R_K$, we can find two points $x_n\in \partial B_{1/n}(x)\cap P_n$ and $y_n\in \partial B_{1/n}(y)\cap P_n$ such that $(x_n,y_n)\in R_K\subset\overline{R_K}$. Clearly, we have $\lim_nx_n=x$ and $\lim_ny_n=y$.

By going to an appropriate subsequence, if necessary, we may assume that $\lim_nP_n=P_\infty$ under Hausdorff distance. The limit $P_\infty$ is a continuum, contains $\{x,y\}$ and is contained in $K$. To finish our proof, we just need to show that $(x,z)\in \overline{R_K}$ for any point $z\in P_\infty\setminus\{x,y\}$, which indicates that $P_\infty\subset \overline{R_K}[x]$.

To this end, we may choose a point $z_n\in P_n$ for all $n\ge1$ with $\frac{1}{n}\le\frac{|x-y|}{3}$ such that $\lim_nz_n=z$. Fix an integer $N>1$ with $\frac{1}{N}\le\frac{1}{2}\min\{|x-z|,|y-z|\}$. Then, for $n\ge N$ we have $B_{1/n}(x)\cap B_{1/n}(z)=B_{1/n}(y)\cap B_{1/n}(z)=\emptyset$. Moreover, for all $n\ge N$ we may find two open arcs $\alpha_n,\beta_n$ in the open annulus $\hat{\mathbb{C}}\setminus\overline{B_{1/n}(x)\cup B_{1/n}(y)}$ disjoint from $K$, each of which connects a point on $\partial B_{1/n}(x)$ to a point on $\partial B_{1/n}(y)$. See the following Figure \ref{conn-fiber}
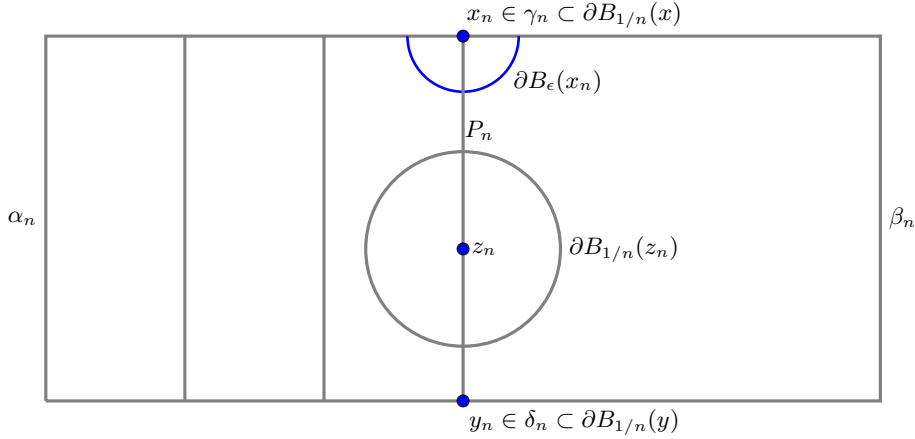
\begin{figure}[ht]
\centering
\vskip -0.25cm
\begin{tikzpicture}[scale=1.618]
\footnotesize
 \pgfmathsetmacro{\xone}{0}
 \pgfmathsetmacro{\xtwo}{16}
 \pgfmathsetmacro{\yone}{0}
 \pgfmathsetmacro{\ytwo}{3.5}

\draw[gray,very thick] (\xone,\yone) -- (\xone+3*\xtwo/7,\yone) --  (\xone+3*\xtwo/7,\yone+6*\ytwo/7) --
(\xone,\yone+6*\ytwo/7) -- (\xone,\yone);

\draw[gray,very thick] (1.5*\xtwo/7,2.5*\ytwo/7) circle (0.35*\xtwo/7);
\draw[line width=1pt,color=blue] (1.3*\xtwo/7,6*\ytwo/7) arc(180:360:0.2*\xtwo/7);

\draw[black] (\xone,3*\ytwo/7) node[anchor=east] {$\alpha_n$};
\draw[black] (3*\xtwo/7,3*\ytwo/7) node[anchor=west] {$\beta_n$};
\draw[black] (1.9*\xtwo/7,6*\ytwo/7) node[anchor=south] {$x_n\in\gamma_n\subset \partial B_{1/n}(x)$};

\draw[black] (1.9*\xtwo/7,0*\ytwo/7) node[anchor=north] {$y_n\in\delta_n\subset \partial B_{1/n}(y)$};

\foreach \p in {.4,.9,1.4}   \draw[gray,very thick] (0.1*\xtwo/7+\p*\xtwo/7,0*\ytwo/7) -- (0.1*\xtwo/7+\p*\xtwo/7,6*\ytwo/7);
\draw[black] (1.475*\xtwo/7,4.45*\ytwo/7) node[anchor=west] {$P_n$};
\draw[black] (1.85*\xtwo/7,2.5*\ytwo/7) node[anchor=west] {$\partial B_{1/n}(z_n)$};
\draw[black] (1.65*\xtwo/7,5.25*\ytwo/7) node[anchor=west] {$\partial B_{\epsilon}(x_n)$};

\draw[fill=blue] (1.5*\xtwo/7,6*\ytwo/7) circle (0.15em);
\draw[fill=blue] (1.5*\xtwo/7,0*\ytwo/7) circle (0.15em);
\draw[fill=blue] (1.5*\xtwo/7,2.5*\ytwo/7) circle (0.15em);
\draw[black] (1.5*\xtwo/7,2.5*\ytwo/7) node[anchor=west] {$z_n$};
\end{tikzpicture}
\vskip -0.75cm
\caption{Relative locations of $\alpha_n,\beta_n,\gamma_n,\delta_n$ and $P_n$.}\label{conn-fiber}
\end{figure}
for relative locations of $\alpha_n, \beta_n, x_n, y_n, z_n$ and $P_n$. Then  $\hat{\mathbb{C}}\setminus\overline{B_{1/n}(x)\cup B_{1/n}(y)}$ is divided by $\overline{\alpha_n\cup\beta_n}$ into two open (topological) disks, whose closures are closed (topological) disks. One of those closed disks, denoted as $D_n$, contains $P_\infty$ and hence all but finitely many of the continua $P_n$. And the boundary $\partial D_n$ equals $\alpha_n\cup\beta_n\cup\gamma_n\cup\delta_n$, where $\gamma_n=D_n\cap\overline{B_{1/n}(x)}$ and $\delta_n=D_n\cap\overline{B_{1/n}(y)}$. See Figure \ref{conn-fiber}. Clearly, we have $x_n\in\gamma_n$ and $y_n\in\delta_n$. Let $B_\epsilon(x_n)$ be an open disk centered at $x_n$ with small enough radius $\epsilon\in(0,1/n)$. Then $D_n\setminus B_\epsilon(x_n)$ is a closed (topological) disk, whose boundary is a simple closed curve, denoted as $J_n$. Let $U_n$ be the open annulus bounded by the two simple closed curves $J_n$ and $\partial B_{1/n}(z_n)$. Then $K\cap\overline{U_n}$ has infinitely many components each of which intersects $\partial B_{1/n}(z_n)$ and $\partial B_\epsilon(x_n)$ at the same time. Then we may find a point $x_n'\in\partial B_\epsilon(x_n)$ and a point $z_n'\in\partial B_{1/n}(z_n)$ with $(x_n',z_n')\in R_K$. Since each of $|x_n-x_n'|$ and $|z_n-z_n'|$ is smaller than $\frac{1}{n}$, we have $\lim_nx_n'=\lim_nx_n=x$ and $\lim_nz_n'=\lim_nz_n=z$. Consequently, we have $(x,z)\in\overline{R_K}$.
\end{proof}

\section{Fibers  of $\overline{R_K}$ have nice behaviour under rational functions}

In this section, $f$ is assumed to be a rational map with degree $d\ge2$. And Theorem \ref{invariant-fiber} consists of two parts: Theorems \ref{fiber-into} and \ref{fiber-onto}. In the following we recall a useful proposition from  \cite{Beardon-91}, which also appears as Lemma 5.7.2 in \cite[p.95]{Beardon91}.
\begin{lemma}\label{finite-monotone}
Let $T\subset\hat{\mathbb{C}}$ be a continuum. Then $f^{-1}(T)$ has at most $d$ components, and each is mapped by $f$ onto $T$.
\end{lemma}

Given a compact $K\subset\hat{\mathbb{C}}$ and the Sch\"onflies relation $\overline{R_K}$. We will relate the fibers of $\overline{R_K}$ to those of $\overline{R_{f^{-1}(K)}}$. The same relation has been obtained when $K$ is assumed unshielded \cite{BCO11,BCO13}.

\begin{theo}\label{fiber-into}
Let $K\subset\hat{\mathbb{C}}$ be a compact set and $R_K$ the Sch\"onflies relation on $K$. Then  the containment $f\left(\overline{R_{f^{-1}(K)}}[x]\right)\subset\overline{R_{K}}[f(x)]$ holds for all $x\in f^{-1}(K)$.
\end{theo}
\begin{proof} Fix $x,y\in f^{-1}(K)$ with $(x,y)\in\overline{R_{f^{-1}(K)}}$. Given an arbitrary number $r>0$, by Theorem \ref{closure} we may find  two infinite sequences $\{x_n\},\{y_n\}\subset f^{-1}(K)$ with $|x_n-x|=|y_n-y|=\frac{r}{n}$ such that $(x_n,y_n)\in R_{f^{-1}(K)}$.
Since $f$ has finitely many critical points, an appropriate choice of $r>0$ will guarantee that all those points $x_n,y_n$ are regular. Therefore, we only need to verify that $(f(x_n),f(y_n))\in\overline{R_{K}}$ for each $n\ge1$.
That is to say, we need to show that for any number $\delta\in\left(0,\frac{1}{2}\left|f\left(x_n\right)-f\left(y_n\right)\right|\right)$ the compact set $X_n:=K\setminus(B_\delta(f(x_n))\cup B_\delta(f(y_n)))$ has infinitely many components intersecting the two circles  $\partial B_\delta(f(x_n)$ and $\partial B_\delta(f(y_n)$ at the same time.

Since $x_n$ and $y_n$ are regular points of the rational map $f$, we may fix a positive number $\epsilon<\frac{1}{2}|x_n-y_n|$ such that (1) $|f(u)-f(v)|<\delta$ for any $u,v\in\hat{\mathbb{C}}$ with $|u-v|<\epsilon$ and that (2) the following two maps are each a homeomorphism: \begin{equation}\label{restriction-1}
f|_{B_\epsilon(x_n)}:B_\epsilon(x_n)\rightarrow f(B_\epsilon(x_n)),\qquad\qquad f|_{B_\epsilon(y_n)}:B_\epsilon(x_n)\rightarrow f(B_\epsilon(y_n))
\end{equation}
Using Theorem \ref{closure} again, we see that $f^{-1}(K)\setminus(B_\epsilon(x_n)\cup B_\epsilon(y_n))$ has infinitely many components $\{P_i\}$ each of which intersects both $\partial B_\epsilon(x_n)$ and $\partial B_\epsilon(y_n)$. Thus each $f(P_i)$ is a continuum and intersects both $f(\partial B_\epsilon(x_n))\subset B_\delta(f(x_n))$ and $f(\partial B_\epsilon(y_n))\subset B_\delta(f(y_n))$. By the well known Boundary Bumping Theorem II~\cite[Theorem 5.6, p74]{Nadler92} 
, every component of $f(P_i)\setminus\left(B_\delta(f(x_n))\bigcup B_\delta(f(y_n))\right)$ intersects either $\partial B_\delta(f(x_n)$ or $\partial B_\delta(f(y_n)$. Combining this fact with the connectedness of $f(P_i)$, we may deduce that $f(P_i)\setminus\left(B_\delta(f(x_n))\bigcup B_\delta(f(y_n))\right)$ has a component $Q_i$ which intersects both $\partial B_\delta(f(x_n))$ and $\partial B_\delta(f(y_n))$.

Clearly, every $Q_i$ is contained in a component of
$X_n
$. Our proof will be completed if only we can show that no component of $X_n$ contains more than $d$ of the above $Q_i$. If on the contrary $Q_{i(1)}, \ldots, Q_{i(d+1)}$ were contained in a single component $T$ of $X_n$, then by Lemma \ref{finite-monotone} we see that the pre-image $f^{-1}(T)\subset \left[f^{-1}(K)\setminus\left(B_\epsilon(x)\cup B_\epsilon(y)\right)\right]$ has $\le d$ components one of which intersects at least two of the pre-images $f^{-1}(Q_{i(j)})$, say $f^{-1}(Q_{i(1)})$ and $f^{-1}(Q_{i(2)})$. Denote that component of $f^{-1}(T)$ by $M$. Then the two components $P_{i(1)}, P_{i(2)}$ of $f^{-1}(K)\setminus(B_\epsilon(x_n)\cup B_\epsilon(y_n))$  are contained in $M$, a connected subset of $f^{-1}(K)\setminus\left(B_\epsilon(x)\cup B_\epsilon(y)\right)$ which is necessarily contained in a single component of $f^{-1}(K)\setminus(B_\epsilon(x_n)\cup B_\epsilon(y_n))$. This is absurd.
\end{proof}

The following corollary is based on Theorem \ref{fiber-into}.

\begin{coro}\label{forward-invariance}
Let $K\subset\hat{\mathbb{C}}$ be compact. Then the core decomposition $\Dc_{f^{-1}(K)}^{PS}$ refines $f^{-1}\left(\Dc_K^{PS}\right)$. Equivalently, each element of $\Dc_{f^{-1}(K)}^{PS}$ is sent by $f$ into an element of $\Dc_{K}^{PS}$.
\end{coro}
\begin{proof}
By Theorem \ref{fiber-into}, every fiber of $\overline{R_{f^{-1}(K)}}$ is sent into a fiber of $\overline{R_K}$, thus is contained in a single element of $f^{-1}\left(\Dc_K^{PS}\right)$, a monotone decomposition of $f^{-1}(K)$. Since the equivalence on $f^{-1}(K)$ corresponding to $\Dc_{f^{-1}(K)}^{PS}$ is the smallest closed equivalence containing $\overline{R_{f^{-1}(K)}}$, we see that $\Dc_{f^{-1}(K)}^{PS}$ refines $f^{-1}\left(\Dc_K^{PS}\right)$.
\end{proof}

\begin{theo}\label{fiber-onto}
Let $K\subset\hat{\mathbb{C}}$ be a compact set and $R_K$ the Sch\"onflies relation on $K$. Then  the equality $f\left(\overline{R_{f^{-1}(K)}}[x]\right)\supset\overline{R_{K}}[f(x)]$ holds for all $x\in f^{-1}(K)$.
\end{theo}
\begin{proof}
We need to show that $f^{-1}(w)\bigcap\overline{R_{f^{-1}K}}[x]\ne\emptyset$ for any $w\in \overline{R_K}[f(x)]$. Using the same argument as those in the proof for Theorem \ref{fiber-into} (see first paragraph), we only need to consider the case that neither $f(x)$ nor $w$ is a critical value.

Given a positive number $\epsilon<\frac{|f(x)-w|}{3}$ such that the restrictions of $f$ to all the component of $f^{-1}\left(\overline{B_\epsilon(f(x))}\right)\bigcup f^{-1}\left(\overline{B_\epsilon(w)}\right)$ are each a homeomorphism. Let $U_1,\ldots, U_d$ be the components of $f^{-1}\left(B_\epsilon(f(x))\right)$. Let $V_1,\ldots, V_d$ be the components of $f^{-1}\left(B_\epsilon(w)\right)$. Then each of the closures $\overline{U_1},\ldots, \overline{U_d},\  \overline{V_1},\ldots, \overline{V_d}$ is a closed topological disk, and those disks are pairwise disjoint.
We may assume that $x\in U_1$.

As $w\in \overline{R_K}[f(x)]$, the compact set $K\setminus(B_\epsilon(f(x))\cup B_\epsilon(w))$ has infinitely many components $\{P_n\}$ that intersect both $\partial B_\epsilon(f(x))$ and $\partial B_\epsilon(w)$. By Lemma \ref{finite-monotone}, every pre-image $f^{-1}(P_n)$ has $\le d$ components, each of which is a component of $X\setminus\left(U_1\cup\cdots\cup U_d\cup V_1\cup\cdots\cup V_d\right)$ that contains a point on some $\partial U_i$ and a point on some $\partial V_j$. In particular, one of those components, denoted as $Q_n$, intersects both $\partial U_1$ and $\partial V_j$. Then there exists one $V_j$, say $V_1$, such that infinitely many of $\{Q_n\}$ intersect both $\partial U_1$ and $\partial V_1$.

Now, we may apply Theorem \ref{good-R} (in Section 2)
and infer that $(x_\epsilon,s_\epsilon)\in R_{f^{-1}(K)}$ for some $x_\epsilon\in\partial U_1$ and some $s_\epsilon\in\partial V_1$. Clearly, we have $\lim\limits_{\epsilon\rightarrow0}x_\epsilon=x$. Moreover, we may choose an appropriate sequence $\epsilon(k)\rightarrow0$ such that $\lim\limits_{k\rightarrow\infty}s_{\epsilon(k)}=s$ for some point $s\in f^{-1}(K)$. Since  $\left(x_{\epsilon(k)},s_{\epsilon(k)}\right)\in R_{f^{-1}(K)}$ for all $k\ge1$, we immediately have $(x,s)\in \overline{R_{f^{-1}(K)}}$.
\end{proof}

\section{The coincidence of two equivalences}\label{coincidence}

This section proves  Theorem \ref{final}, in which we always assume that $K\subset\hat{\mathbb{C}}$ is a compactum and $f$ a rational map with degree $d\ge2$. Let $\sim$ be the smallest closed equivalence containing the Sch\"onflies relation $R_K$. By  \cite[Theorem 7]{LLY-2017}, the core decomposition $\Dc_K^{PS}$ is formed by the classes $[x]_\sim=\{z\in K: x\sim z\}$. Similarly, the core decomposition $\Dc_L^{PS}$ of $L=f^{-1}(K)$  is formed by the classes $[x]_\asymp=\{z\in f^{-1}(K): x\sim z\}$, where $\asymp$ is the smallest closed equivalence containing $R_L$, the Sch\"onflies relation defined on $L$.

Recall that the equivalence $\approx$ on $K$ is defined by requiring that $x\approx y$ if and only if 
$\pi_L\left(f^{-1}(x)\right)=\pi_L\left(f^{-1}(y)\right)$. Since $\pi_L(z)=[z]_\asymp$ for all $z\in L$, the following is immediate:
\begin{equation}\label{new-eq}
\forall\ x,y\in K,\quad x\approx y \ \Leftrightarrow \bigcup_{f(u)=x}[u]_\asymp=\bigcup_{f(v)=y}[v]_\asymp.
\end{equation}
Moreover, for any $u\in L$ it is routine to check that $[f(u)]_\approx$ is a subset of $f\left([u]_\asymp\right)$. By  Corollary \ref{forward-invariance}, we have $f\left([u]_\asymp\right)\subset [f(u)]_\sim$, from which follows the containment $[f(u)]_\approx\subset [f(u)]_\sim$.

The rest of this section continues to obtain $[f(u)]_\sim\subset [f(u)]_\approx$, which proves Theorem \ref{final}. Since $\sim$ is the smallest closed equivalence containing $R_K$, we only need to show the following.

\begin{theo}\label{comparing}
The equivalence $\approx$ on $K$ is closed and contains $\overline{R_K}$.
\end{theo}
\begin{proof}
Suppose that $x_n\approx y_n$ and $(x_n,y_n)\rightarrow(x,y)$, we need to verify that $x\approx y$. By definition of $\approx$, for any $n\ge1$ we have
\begin{equation}\label{eq-sequence}
\bigcup_{f(u)=x_n}[u]_\asymp=\bigcup_{f(v)=y_n}[v]_\asymp.
\end{equation}
Given $s\in f^{-1}(x)$. Since $f$ is a rational map, for any $k\ge1$ we can check that $f(B(s,1/k))$ is an open neighborhood of $x$  hence contains all but finitely many $x_n$. Let $u_{n(k)}\in f^{-1}(x_{n(k)})$ be a point satisfying $\left|u_{n(k)}-s\right|<\frac{1}{k}$. Clearly, we have $\lim\limits_{k\rightarrow\infty}u_{n(k)}=s$. By Equation \ref{eq-sequence}, we can find $v_{n(k)}\in f^{-1}\left(y_{n(k)}\right)$ with $u_{n(k)}\asymp v_{n(k)}$. Going to an appropriate subsequence, if necessary, we may assume that $\lim\limits_{k\rightarrow\infty}v_{n(k)}=t$. Then $f(t)=\lim\limits_{k\rightarrow\infty}f\left(v_{n(k)}\right)=y$. Clearly, we have $s\asymp t$ and \[\displaystyle [s]_\asymp\subset \left(\bigcup_{v\in f^{-1}(y)}[v]_\asymp\right).
\]
Repeat the same argument for any given $t\in f^{-1}(y)$, we can show that
$[t]_\asymp$ is a subset of $\bigcup_{u\in f^{-1}(x)}[u]_\asymp$.
Therefore, we have shown $x\approx y$. We will continue to show that  $\approx$ contains $\overline{R_K}$.

Given $x\in K$ and $y\in\overline{R_K}[x]$. By Theorem \ref{fiber-onto}, the equality $f\left(\overline{R_L}[u]\right)=\overline{R_{K}}[x]$ holds  for any $u\in f^{-1}(x)$. Thus we may fix a point $v\in \overline{R_L}[u]$ with  $f(v)=y$. As $[u]_\asymp=[v]_\asymp$, we have verified that $\bigcup_{u\in f^{-1}(x)}[u]_\asymp$ is a subset of
$\bigcup_{v\in f^{-1}(y)}[v]_\asymp$. Since the inverse containment may be verified by the same argument, we already have $x\approx y$.
\end{proof}

\section{Can we say something about the Mandelbrot set ?}\label{M}

This section proves Theorem \ref{E}. And we start from some standard notions of complex dynamics.

Given a number $c\in \mathbb{C}$ the filled Julia set $K_c$ of $f_c(z)=z^2+c$ consists of the all the points $z\in\mathbb{C}$ whose orbit $\left\{f_c^k(z): k\right\}$ is bounded. The Mandelbrot set $\M$ consists of the parameters $c$ such that $K_c$ is connected. It is well known that there is a conformal isomorphism $\Phi$ of $\hat{\mathbb{C}}\setminus\M$ onto $\hat{\mathbb{C}}\setminus\overline{\mathbb{D}}=\{z\in\hat{\mathbb{C}}: |z|>1\}$, fixing $\infty$ and having real derivative at $\infty$ \cite{DouadyHubbard82}. We will refer to \cite{CG-93} and \cite{Hubbard93} for the basic notions and for the fundamental properties of $\M$:
\begin{itemize}
\item For any $t\in[0,1)$ the pre-image of $\left\{re^{2\pi{\bf i}t}: r>1\right\}$ under $\Phi$ is called the {\em external ray of $\M$ at external angle $t$} and is denoted as $\Rc^\M_t$. If $\lim_{r\rightarrow1}\Phi^{-1}\left(re^{2\pi{\bf i}t}\right)=c$ then $c\in\partial\M$. In such a case, we say that $\Rc^\M_t$ lands at $c$.

\item For each $\lambda$, $|\lambda|\le1$, there is a unique $c=c(\lambda)$ such that $f_c(z)$ has a fixed point with multiplier $\lambda$. The parameters $c$ for which $f_c$ has an attracting fixed point form a cardioid $H_0\subset\M$, and $\partial H_0\subset\partial\M$ \cite[p.126, Theorem 1.3]{CG-93}. This set $H_0$ is called the major cardioid.

\item A parameter $c\in\M$ is hyperbolic if $f_c(z)$ has an attractive cycle. In particular, if $f_c$ has an attracting cycle of period $m$ then $c$ lies in the interior $\M^o$; if $H$ is the component of $\M^o$ containing $c$ then $f_{\eta}$ has an attracting cycle $z_1(\eta),\ldots,z_m(\eta)$ for all $\eta\in H$, where each $z_j(\eta)$ depends analytically on $\eta$ \cite[p.127, Theorem 1.4]{CG-93}. A component of $\M^o$ containing a hyperbolic parameter is called a hyperbolic component.

\item Every $c\in\partial\M$ can be approximated by hyperbolic parameters \cite[p.129, Theorem 1.5]{CG-93}.

\item If $H$ is a hyperbolic component then the multiplier map $\phi_H$ sending $c\in H$ to the multiplier  of the attracting cycle of $f_c$ is a conformal isomorphism of $H$ onto the unit disk $\mathbb{D}=\{z: |z|<1\}$. It extends continuously to $\partial H$ and maps $\overline{H}$ homeomorphically onto the closed disk $\overline{\mathbb{D}}$ \cite[p.134, Theorem 2.1]{CG-93}. The extended map is also denoted as $\phi_H$ and the two points $\phi_H^{-1}(0), \phi_H^{-1}(1)$ are respectively called the {\em center} and the {\em root} of $H$.

\item Given a hyperbolic component $H$ and $t=p/q\in[0,1)$ with  relatively prime integers $q>p$, there are external rays $\Rc^\M_{t_-}$ and     $\Rc^\M_{t_+}$ landing at $c_{H,t}=\phi_{H}^{-1}(e^{2\pi{\bf i}t})$ \cite[p.154, Theorem 7.2]{CG-93}. There is one exception, when $H=H_0$ and $t=0$, then $c_{H,0}=\frac{1}{4}$ and there is exactly one external ray $\Rc^\M_{0}$ landing at $c_{H,0}$. The union $\Gamma_{H,t}=\{c_{H,t}\}\cup\{\infty\}\cup\Rc_{t_-}^\M\cup\Rc_{t_+}^\M$ is a simple closed curve and the component of $\hat{\mathbb{C}}\setminus\Gamma_{H,t}$ not containing $0$ is denoted as $W_{H,p/q}$; moreover, $L_{H,p/q}=\M\cap\overline{W_{H,p/q}}$ is called the {\em $p/q$-limb of $H$} \cite[p.477]{Hubbard93}.
    If $H$ is the major cardioid we put $W_{H,0}=\hat{\mathbb{C}}\setminus\{r\in\mathbb{R}: r\ge\frac{1}{4}\}$.
\end{itemize}
The following lemma is from \cite[p.277,Proposition 4.2]{Hubbard93} and plays a crucial role in this section.
\begin{lemma}\label{Key-Fact}
Every point of $\M\cap W_{H,0}$ is either in $\overline{H}$ or in one of the limbs $L_{H,p/q}$. There exists a function $\eta_H: \mathbb{N}\rightarrow\mathbb{R}$ with $\eta_H(q)\rightarrow 0$ as $q\rightarrow\infty$, such that the diameter ${\rm diam} L_{H,p/q}\le\eta_H(q)$.
\end{lemma}

Besides the standard notions recalled as above, the proof for Theorem \ref{E} is also benefited from two things:
(1) the techniques used by Hubbard in \cite[Corollary 4.4 and 4.5]{Hubbard93}, and
(2) a newly found approach \cite[Theorem 2.3]{LY-2017} to construct from below the core decomposition $\Dc_K^{PS}$ of a full continuum $K\subset\hat{\mathbb{C}}$. Actually, the result of \cite[Theorem 2.3]{LY-2017} provides a very simple aspect about the structure of the elements of $\Dc_K^{PS}$. We cite it here as a lemma.

\begin{lemma}\label{LY}
Given a full continuum $K\subset\hat{\mathbb{C}}$ and $d\in\Dc_K^{PS}$. If $x\ne y\in \partial K$ lie in $d$ there exist countably many prime end impressions of $\hat{\mathbb{C}}\setminus K$ whose union is connected and contains $\{x,y\}$.
\end{lemma}

The following lemma relies on Lemma \ref{Key-Fact} and is helpful when we prove Theorem \ref{E}.
\begin{lemma}\label{simple-impression}
Let $x$ be a point on the boundary of a hyperbolic component $H$ of $\M^o$. Let $Imp(\theta)$ be a prime end impression of $\mathbb{C}\setminus \M$, for some $\theta\in[0,1)$, such that $x\in Imp(\theta)$. Then $Imp(\theta)=\{x\}$, except when $H$ is primitive and $x$ is the root of $H$.
\end{lemma}
\begin{proof}
Recall that ${\rm Imp}(\theta)$ consists of all the points $y$ such that there exist a sequence of points $z_n$ with $|z_n|>1$ and $z_n\rightarrow e^{2\pi{\bf i}\theta}$ which satisfy $\lim\limits_{n\rightarrow\infty}\Phi^{-1}(z_n)=y$. Since $x\in\partial H$, we may fix $t\in[0,1)$ with $x=\phi_H^{-1}\left(e^{2\pi{\bf i}t}\right)$ and rational numbers $r_1<t<r_2$ with $r_2-r_1$ very small. When $t=0$ we choose $0<r_2<r_1<1$ such that $r_2$ and $1-r_1$ are both very small. Now, we follow the ideas in \cite[Corollary 4.4 and 4.5]{Hubbard93} and consider the region $V_1\subset\mathbb{D}$ that is cut off by the chord between $e^{2\pi{\bf i}r_1}$ and $e^{2\pi{\bf i}r_2}$. Let $N_1$ be the union of $\phi_H^{-1}(V_1)$ and all the limbs $L_{H,p/q}$ that intersects $\phi_H^{-1}(V_1)$.

If $t$ is irrational then $Imp(\theta)$ lies in $N_1$. By Lemma \ref{Key-Fact}, the diameter of $N_1$ tends to zero as $|r_2-r_1|$ tends to zero; Therefore, we have $Imp(\theta)=\{x\}$. If $x$ is the root of the major cardioid $H_0$, the same argument verifies $Imp(\theta)=\{x\}$.

If $t$ is rational and if $x$ is not the root of $H$ when $H$ is primitive, then $x$ lies on the boundary of another hyperbolic component, say $U$. Then $\{x\}=\overline{H}\cap\overline{U}$. We may rename $H$ and $U$, if necessary, and assume that $t=0$ and that $x=\phi_U^{-1}\left(e^{2\pi{\bf i}s}\right)$ for a rational number $s\in(0,1)$. Under this setting, we have $H\subset L_{U,s}$. Now we may fix two rational numbers $r_3<s<r_4$ with $r_4-r_3$ very small and consider the region $V_2\subset\mathbb{D}$ that is cut off by the chord between   $e^{2\pi{\bf i}r_3}$ and $e^{2\pi{\bf i}r_4}$. If $N_2$ is the union of $\phi_U^{-1}(V_2)$ with all the limbs $L_{U,p/q}$ except for $L_{U,s}$ that intersect $\phi_U^{-1}(V_2)$, then the impression $Imp(\theta)$ is contained in $N_2$. Again by Lemma \ref{Key-Fact}, the diameter of $N_2$ tends to zero as $\max\left\{|r_1-r_2|,|r_3-r_4|\right\}\rightarrow 0$. This indicates that $Imp(\theta)=\{x\}$.
\end{proof}

\begin{rema}
In the above Lemma \ref{simple-impression} if $H$ is the major cardioid and $x=\frac14$ then $\theta=0$ and $Imp(0)=\{x\}$. Therefore, in Theorem \ref{E} we can conclude that $\frac14\in \lambda_\M^{-1}(0)$. The proof for this is just similar to that given in the following arguments.
\end{rema}

\begin{proof}[\bf Proof for Theorem E]
Given a hyperbolic component $H$ and a point $x$ on $\partial H$, which is not the root of $H$ when $H$ is primitive. Assume that $x=\phi_H^{-1}\left(e^{2\pi{\bf i}t}\right)$, where $\phi_H: H\rightarrow\mathbb{D}$ is the multiplier map. We will verify  that the element $D(x)$ of $\Dc_\M^{PS}$ containing $x$ is equal to $\{x\}$.

To this end, we apply Lemma \ref{LY} to $D(x)$ and obtain for any $y\in D(x)$ a countable union  $N_y=\bigcup_k{\rm Imp}(\theta_k)$ of prime end impressions, which is a connected set containing $\{x,y\}$.
Applying Lemma \ref{simple-impression} to $Imp(\theta_k)$, we see that every $Imp(\theta_k)$ consists of a single point $y_k\in\partial H$ whenever $Imp(\theta_k)\cap\overline{H}\ne\emptyset$. From this we see that $N_y\cap\overline{H}$ is a countable set. Since a countable set is not connected, we only need to conclude that $Imp(\theta_k)\cap\overline{H}\ne\emptyset$ for all $k\ge1$.

Assume on the contrary that some impression $Imp(\theta_k)$ lies entirely in a limb $L_{H,r}$ and does not intersect $\overline{H}$. Then there are two possibilities: either $x\notin L_{H,r}$ or $x\in L_{H,r}$.

In the former case we may apply Lemma \ref{Key-Fact} to $H$ and fix two points $x_1,x_2\in(\partial H\setminus N_y)$ that are off every limb $L_{H,p/q}$, such that $\{x_1,x_2\}$ separates $x$ from $\phi_H^{-1}\left(e^{2\pi{\bf i}r}\right)$ on $\partial H$. Let $\alpha,\beta$ denote the two components of $\partial H\setminus\{x_1,x_2\}$ with $x\in\alpha$ and $\phi_H^{-1}\left(e^{2\pi{\bf i}r}\right)\in\beta$. Let $A$ be the union of $\alpha$ with all the limbs $L_{H,p/q}$ intersecting $\alpha$. Let $B$ be the union of $\beta$ with all the limbs $L_{H,p/q}$ that intersect $\beta$. The second part of Lemma \ref{Key-Fact} implies that $\overline{A}\cap\overline{B}=\{x_1,x_2\}$. That is to say, the two sets $A$ and $B$ are separated. This contradicts the connectedness of $N_y$, which intersects $A$ and $B$ both.

In the latter case, we Apply Lemma \ref{Key-Fact} to $L_{H,r}$ and find two points $x_1,x_2\in(\partial L_{H,r}\setminus N_y)$ that are off every limb  of $L_{H,r}$, such that $\{x_1,x_2\}$ separates $x$ from $\phi_{L_{H,r}}^{-1}\left(e^{2\pi{\bf i}0}\right)$ on $\partial L_{H,r}$. We may repeat the argument in the previous paragraph and obtain a contradiction accordingly.
\end{proof}

\section{Examples and Remarks Related to Theorem \ref{invariance}}\label{ex}

We firstly give three examples and a remark concerning FitzGerald-Swingle's core decomposition of continua with semmi-locally connected hyperspace.

\begin{exam}\label{no-CD-ex}
We will construct a continuum $K\subset\mathbb{R}^3$
of which the Peano model does not exist. Let $\mathcal{C}\subset[0,1]$ be the Cantor ternary set. Let $K$ be the union of three compact sets: $\mathcal{C}\times[0,1]^2$, $[0,1]\times\{0\}\times[0,1]$, and $[0,1]^2\times\{0\}$. The following Figure \ref{no-CD} gives a simplified depiction. The left part represents $\mathcal{C}\times[0,1]^2$ and the right part the union of $[0,1]\times\{0\}\times[0,1]$ with $[0,1]^2\times\{0\}$.
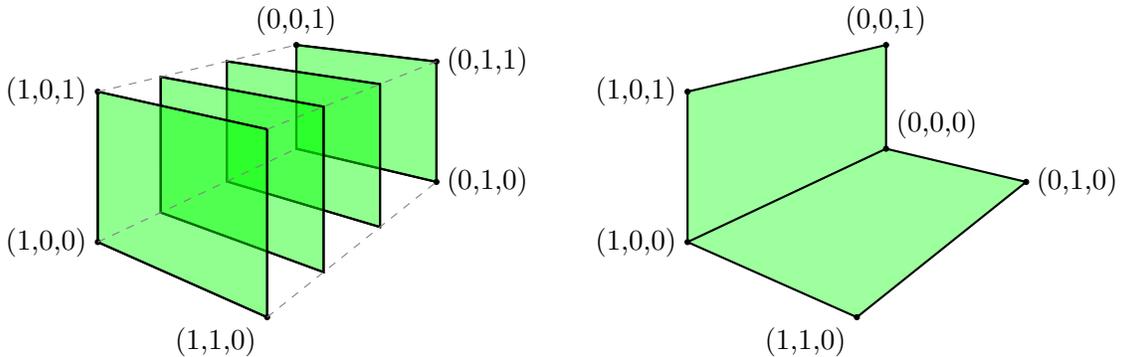
\begin{figure}[ht]
\vskip-0.0cm

\begin{center}
\begin{tikzpicture}[x=1cm,y=1cm,scale=1.25]

	\coordinate (P1) at (-9cm,2cm); 
	\coordinate (P2) at (5cm,2cm); 

	\coordinate (A1) at (0cm,0cm); 
	\coordinate (A2) at (0cm,-2cm); 

	\coordinate (A3) at ($(P1)!.8!(A2)$); 
	\coordinate (A4) at ($(P1)!.8!(A1)$);

	\coordinate (A6) at ($(P2)!.64!(A2)$);
	\coordinate (A5) at ($(P2)!.64!(A1)$);

	\coordinate (A8) at
	  (intersection cs: first line={(A5) -- (P1)},
			    second line={(A4) -- (P2)});
	\coordinate (A7) at
	  (intersection cs: first line={(A6) -- (P1)},
			    second line={(A3) -- (P2)});


%


	\foreach \i in {0,1,2,3}  	\foreach \j in {0,2}
	{
	\coordinate (B2) at ($(P2)!0.64+0.12*\i!(A2)$);
	\coordinate (B1) at ($(P2)!0.64+0.12*\i!(A1)$);

	\coordinate (B4) at
	  (intersection cs: first line={(B1) -- (P1)},
			    second line={(A4) -- (P2)});
	\coordinate (B3) at
	  (intersection cs: first line={(B2) -- (P1)},
			    second line={(A3) -- (P2)});
\fill[green,opacity=0.25] (B1) -- (B2) -- (B3) -- (B4) -- cycle; 
\draw[black,thick] (B1) -- (B2) -- (B3) -- (B4) -- cycle;
	}



\draw[fill=black] (A4) circle (0.07em)  node[left] {(1,0,1)};
\draw[fill=black] (A3) circle (0.07em)  node[left] {(1,0,0)};
\draw[fill=black] (A2) circle (0.07em)  node[below left] {(1,1,0)};
\draw[fill=black] (A6) circle (0.07em)  node[right] {(0,1,0)};
\draw[fill=black] (A5) circle (0.07em)  node[right] {(0,1,1)};
\draw[fill=black] (A8) circle (0.07em)  node[above] {(0,0,1)};
\draw[gray,dashed] (A1) -- (A5);
\draw[gray,dashed] (A2) -- (A6);
\draw[gray,dashed] (A4) -- (A8);
\draw[gray,dashed] (A3) -- (A7);
\end{tikzpicture}
\hskip 0.5cm
\begin{tikzpicture}[x=1cm,y=1cm,scale=1.25]

	\coordinate (P1) at (-9cm,2cm); 
	\coordinate (P2) at (5cm,2cm); 

	\coordinate (A1) at (0cm,0cm); 
	\coordinate (A2) at (0cm,-2cm); 

	\coordinate (A3) at ($(P1)!.8!(A2)$); 
	\coordinate (A4) at ($(P1)!.8!(A1)$);

	\coordinate (A6) at ($(P2)!.64!(A2)$);
	\coordinate (A5) at ($(P2)!.64!(A1)$);

	\coordinate (A8) at
	  (intersection cs: first line={(A5) -- (P1)},
			    second line={(A4) -- (P2)});
	\coordinate (A7) at
	  (intersection cs: first line={(A6) -- (P1)},
			    second line={(A3) -- (P2)});


	\fill[green!38.2] (A2) -- (A3) -- (A7) -- (A6) -- cycle; 
	\fill[green!38.2] (A3) -- (A4) -- (A8) -- (A7) -- cycle; 


\draw[fill=black] (A4) circle (0.07em)  node[left] {(1,0,1)};
\draw[fill=black] (A3) circle (0.07em)  node[left] {(1,0,0)};
\draw[fill=black] (A2) circle (0.07em)  node[below left] {(1,1,0)};
\draw[fill=black] (A6) circle (0.07em)  node[right] {(0,1,0)};
\draw[fill=black] (A8) circle (0.07em)  node[above] {(0,0,1)};
\draw[fill=black] (A7) circle (0.07em)  node[above right] {(0,0,0)};
 \draw[thick,black] (A3) -- (A7);
\draw[thick,black] (A2) -- (A6);
\draw[thick,black] (A4) -- (A8);
\draw[thick,black] (A2) -- (A3);
\draw[thick,black] (A6) -- (A7); \draw[thick,black] (A3) -- (A4); \draw[thick,black] (A7) -- (A8);
\end{tikzpicture}
\end{center}
\vskip -1.0cm
\caption{A continuum in $\mathbb{R}^3$ which has no Peano model.}\label{no-CD}
\vskip -0.25cm
\end{figure}
It is routine to check that $K$ is a continuum and is semi-locally connected everywhere. Consider the projections $\pi_1(x_1,x_2,x_3)=(x_1,x_2)$ and $\pi_2(x_1,x_2,x_3)=(x_1,x_3)$ from $K$ onto $[0,1]^2$.
One may verify that the two collections of pre-images $\Dc_i=\{\pi_i^{-1}(x): x\in[0,1]^2\}$ are each a monotone decomposition of $K$, whose elements are either single points or segments. Moreover, the only decomposition of $K$ finer than both $\Dc_1$ and $\Dc_2$ is the decomposition into single points. Therefore, $K$ does not have a Peano model, since it is not locally connected.
\end{exam}

\begin{exam}\label{invariance-destroyed} Let $\mathcal{C}$ be Cantor's ternary set and $X$ Sierpinski's carpet. Let $K$ be the image of $X$ under the map $\displaystyle (t,r)\mapsto r^2e^{2\pi{\bf i}t}$.
Let $Y$ be the union of $X$ and the product $[1,2]\times\mathcal{C}$. Let $L$ be the image of $X\cup Y$ under the map $(s,r)\mapsto re^{\pi{\bf i}s}$. Then $L\subset\{z: |z|\le1/2\}$ and $K\subset\{z: |z|\le1/\sqrt{2}\}$ satisfy $f(L)=K$ and $L\nsubseteq f^{-1}(K)$, where $f(z)=z^2$; moreover, every element of $\Dc_K^{PS}$ is a singleton while $\Dc_L^{PS}$ has uncountably many elements each of which is a non-degenerate Jordan arc. Those arcs are mapped to circle centered at the origin. See the following Figure \ref{variant-CD}.
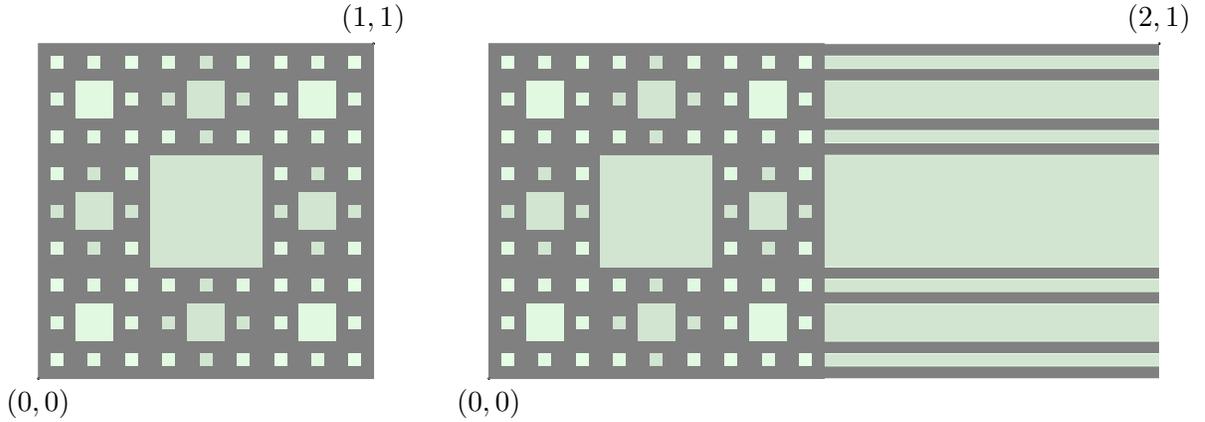
\begin{figure}[ht]
\vskip -0.25cm
\begin{center}
\begin{tabular}{cc}
\begin{tikzpicture}[x=1cm,y=1cm,scale=0.55]
\fill[gray] (0,0) -- (8.1,0) -- (8.1,8.1) -- (0,8.1) -- cycle;
\draw[fill=black] (8.1,8.1) circle (0.05em)  node[above] {$(1,1)$};
\draw[fill=black] (0,0) circle (0.05em)  node[below] {$(0,0)$};

	\foreach \i in {0,1,2}
	{
\fill[green!10,opacity=0.80] (2.7*\i+0.9,0.9) -- (2.7*\i+0.9,1.8) -- (2.7*\i+1.8,1.8) -- (2.7*\i+1.8,0.9) -- cycle;
\fill[green!10,opacity=0.80] (2.7*\i+0.9,6.3) -- (2.7*\i+0.9,7.2) -- (2.7*\i+1.8,7.2) -- (2.7*\i+1.8,6.3) -- cycle;

\fill[green!10,opacity=0.80] (0.9,2.7*\i+0.9) -- (1.8,2.7*\i+0.9) -- (1.8,2.7*\i+1.8) -- (0.9,2.7*\i+1.8) -- cycle;
\fill[green!10,opacity=0.80] (6.3,2.7*\i+0.9) -- (7.2,2.7*\i+0.9) -- (7.2,2.7*\i+1.8) -- (6.3,2.7*\i+1.8) -- cycle;
	}

\foreach \i in {0,1,2}
\foreach \j in {0,1,2}
\foreach \k in {0,2}
{
\fill[green!10,opacity=0.80] (2.7*\k+0.9*\i+0.3,2.7*\j+0.3) -- (2.7*\k+0.9*\i+0.3,2.7*\j+0.6) -- (2.7*\k+0.9*\i+0.6,2.7*\j+0.6) -- (2.7*\k+0.9*\i+0.6,2.7*\j+0.3) -- cycle;
\fill[green!10,opacity=0.80] (2.7*\k+0.9*\i+0.3,2.7*\j+2.1) -- (2.7*\k+0.9*\i+0.3,2.7*\j+2.4) -- (2.7*\k+0.9*\i+0.6,2.7*\j+2.4) -- (2.7*\k+0.9*\i+0.6,2.7*\j+2.1) -- cycle;
\fill[green!10,opacity=0.80] (2.7*\k+0.3,2.7*\j+0.9*\i+0.3) -- (2.7*\k+0.6,2.7*\j+0.9*\i+0.3) -- (2.7*\k+0.6,2.7*\j+0.9*\i+0.6) -- (2.7*\k+0.3,2.7*\j+0.9*\i+0.6) -- cycle;
\fill[green!10,opacity=0.80] (2.7*\k+2.1,2.7*\j+0.9*\i+0.3) -- (2.7*\k+2.4,2.7*\j+0.9*\i+0.3) -- (2.7*\k+2.4,2.7*\j+0.9*\i+0.6) -- (2.7*\k+2.1,2.7*\j+0.9*\i+0.6) -- cycle;

\fill[green!10,opacity=0.80] (2.7*\j+0.3,2.7*\k+0.9*\i+0.3) -- (2.7*\j+0.6,2.7*\k+0.9*\i+0.3) -- (2.7*\j+0.6,2.7*\k+0.9*\i+0.6) -- (2.7*\j+0.3,2.7*\k+0.9*\i+0.6) -- cycle;
\fill[green!10,opacity=0.80] (2.7*\j+2.1,2.7*\k+0.9*\i+0.3) -- (2.7*\j+2.4,2.7*\k+0.9*\i+0.3) -- (2.7*\j+2.4,2.7*\k+0.9*\i+0.6) -- (2.7*\j+2.1,2.7*\k+0.9*\i+0.6) -- cycle;
\fill[green!10,opacity=0.80] (2.7*\j+0.9*\i+0.3,2.7*\k+0.3) -- (2.7*\j+0.9*\i+0.3,2.7*\k+0.6) -- (2.7*\j+0.9*\i+0.6,2.7*\k+0.6) -- (2.7*\j+0.9*\i+0.6,2.7*\k+0.3) -- cycle;
\fill[green!10,opacity=0.80] (2.7*\j+0.9*\i+0.3,2.7*\k+2.1) -- (2.7*\j+0.9*\i+0.3,2.7*\k+2.4) -- (2.7*\j+0.9*\i+0.6,2.7*\k+2.4) -- (2.7*\j+0.9*\i+0.6,2.7*\k+2.1) -- cycle;
}

\draw[gray] (0,0) -- (8.1,0) -- (8.1,8.1) -- (0,8.1) -- cycle;
\fill[green!10,opacity=0.80] (2.7,2.7) -- (5.4,2.7) -- (5.4,5.4) -- (2.7,5.4) -- cycle;

\end{tikzpicture}
&
\begin{tikzpicture}[x=1cm,y=1cm,scale=0.55]
\fill[gray] (0,0) -- (8.1,0) -- (8.1,8.1) -- (0,8.1) -- cycle;
\fill[gray] (8.1,0) -- (16.2,0) -- (16.2,8.1) -- (8.1,8.1) -- cycle;
\draw[fill=black] (16.2,8.1) circle (0.05em)  node[above] {$(2,1)$};
\draw[fill=black] (0,0) circle (0.05em)  node[below] {$(0,0)$};

	\foreach \i in {0,1,2}
	{
\fill[green!10,opacity=0.80] (2.7*\i+0.9,0.9) -- (2.7*\i+0.9,1.8) -- (2.7*\i+1.8,1.8) -- (2.7*\i+1.8,0.9) -- cycle;
\fill[green!10,opacity=0.80] (2.7*\i+0.9,6.3) -- (2.7*\i+0.9,7.2) -- (2.7*\i+1.8,7.2) -- (2.7*\i+1.8,6.3) -- cycle;

\fill[green!10,opacity=0.80] (0.9,2.7*\i+0.9) -- (1.8,2.7*\i+0.9) -- (1.8,2.7*\i+1.8) -- (0.9,2.7*\i+1.8) -- cycle;
\fill[green!10,opacity=0.80] (6.3,2.7*\i+0.9) -- (7.2,2.7*\i+0.9) -- (7.2,2.7*\i+1.8) -- (6.3,2.7*\i+1.8) -- cycle;
	}

\foreach \i in {0,1,2}
\foreach \j in {0,1,2}
\foreach \k in {0,2}
{
\fill[green!10,opacity=0.80] (2.7*\k+0.9*\i+0.3,2.7*\j+0.3) -- (2.7*\k+0.9*\i+0.3,2.7*\j+0.6) -- (2.7*\k+0.9*\i+0.6,2.7*\j+0.6) -- (2.7*\k+0.9*\i+0.6,2.7*\j+0.3) -- cycle;
\fill[green!10,opacity=0.80] (2.7*\k+0.9*\i+0.3,2.7*\j+2.1) -- (2.7*\k+0.9*\i+0.3,2.7*\j+2.4) -- (2.7*\k+0.9*\i+0.6,2.7*\j+2.4) -- (2.7*\k+0.9*\i+0.6,2.7*\j+2.1) -- cycle;
\fill[green!10,opacity=0.80] (2.7*\k+0.3,2.7*\j+0.9*\i+0.3) -- (2.7*\k+0.6,2.7*\j+0.9*\i+0.3) -- (2.7*\k+0.6,2.7*\j+0.9*\i+0.6) -- (2.7*\k+0.3,2.7*\j+0.9*\i+0.6) -- cycle;
\fill[green!10,opacity=0.80] (2.7*\k+2.1,2.7*\j+0.9*\i+0.3) -- (2.7*\k+2.4,2.7*\j+0.9*\i+0.3) -- (2.7*\k+2.4,2.7*\j+0.9*\i+0.6) -- (2.7*\k+2.1,2.7*\j+0.9*\i+0.6) -- cycle;

\fill[green!10,opacity=0.80] (2.7*\j+0.3,2.7*\k+0.9*\i+0.3) -- (2.7*\j+0.6,2.7*\k+0.9*\i+0.3) -- (2.7*\j+0.6,2.7*\k+0.9*\i+0.6) -- (2.7*\j+0.3,2.7*\k+0.9*\i+0.6) -- cycle;
\fill[green!10,opacity=0.80] (2.7*\j+2.1,2.7*\k+0.9*\i+0.3) -- (2.7*\j+2.4,2.7*\k+0.9*\i+0.3) -- (2.7*\j+2.4,2.7*\k+0.9*\i+0.6) -- (2.7*\j+2.1,2.7*\k+0.9*\i+0.6) -- cycle;
\fill[green!10,opacity=0.80] (2.7*\j+0.9*\i+0.3,2.7*\k+0.3) -- (2.7*\j+0.9*\i+0.3,2.7*\k+0.6) -- (2.7*\j+0.9*\i+0.6,2.7*\k+0.6) -- (2.7*\j+0.9*\i+0.6,2.7*\k+0.3) -- cycle;
\fill[green!10,opacity=0.80] (2.7*\j+0.9*\i+0.3,2.7*\k+2.1) -- (2.7*\j+0.9*\i+0.3,2.7*\k+2.4) -- (2.7*\j+0.9*\i+0.6,2.7*\k+2.4) -- (2.7*\j+0.9*\i+0.6,2.7*\k+2.1) -- cycle;
}

\fill[green!10,opacity=0.80] (8.1,2.7) -- (16.2,2.7) -- (16.2,5.4) -- (8.1,5.4)-- cycle;

\foreach \i in {0,2}
{
\fill[green!10,opacity=0.80] (8.1,0.9+2.7*\i) -- (16.2,0.9+2.7*\i) -- (16.2,1.8+2.7*\i) -- (8.1,1.8+2.7*\i)-- cycle;
}

\foreach \i in {0,2}  \foreach \j in {0,2}
{
\fill[green!10,opacity=0.80] (8.1,0.3+0.9*\j+2.7*\i) -- (16.2,0.3+0.9*\j+2.7*\i) -- (16.2,0.6+0.9*\j+2.7*\i) -- (8.1,0.6+0.9*\j+2.7*\i)-- cycle;
}

\foreach \i in {0,2}
\foreach \j in {0,1,2,3}
\foreach \k in {0,2}
{
\fill[green!10,opacity=0.80] (8.1,0.3*\j+0.9*\i+2.7*\k) -- (16.2,0.3*\j+0.9*\i+2.7*\k);
}

\draw[gray] (0,0) -- (8.1,0) -- (8.1,8.1) -- (0,8.1) -- cycle;
\fill[green!10,opacity=0.80] (2.7,2.7) -- (5.4,2.7) -- (5.4,5.4) -- (2.7,5.4) -- cycle;

\end{tikzpicture}

\end{tabular}

\end{center}\vskip -0.5cm
\caption{Finite approximations of $X$ and $Y$.}\label{variant-CD}
\end{figure}
\end{exam}

\begin{exam}\label{slc-bad}
Let $f(z)=z^2$. We will construct a continuum $K\subset\mathbb{C}$ such that the decomposition $\Dc_K^{PS}$ in Theorem \ref{invariance} can not be replaced by $\Dc_K^{SLC}$, the core decomposition of $K$ with respect to the semi-locally connected property. The existence of such a core decomposition has been obtained in \cite[Theorem 2.7]{FitzGerald67}. See Remark \ref{semi-lc} for more details. Let $\mathcal{C}$ be the Cantor ternary set and $K=\{re^{2\pi{\bf i}t}: r\in\mathcal{C}, 0\le t\le1\}\cup[0,1]$. Then $f^{-1}(K)=\{\sqrt{r}e^{2\pi{\bf i}t}: r\in\mathcal{C}, 0\le t\le1\}\cup[-1,1]$ and $K$ are both continua in the plane. Clearly, $f^{-1}(K)$ is semi-locally connected while $K$ is not. Since $\Dc_K^{SLC}$ has uncountably many elements that are concentric circles centered with radius $r$ for some $r\in\mathcal{C}$, the inclusion $f(\{x\})\in\Dc_K^{SLC}$ does not hold for every element $\{x\}\in\Dc_L^{SLC}$ with $|x|\in\mathcal{C}\setminus\{0\}$. However, it is still true that $f$ sends every element of $\Dc_L^{SLC}$ into an element of $\Dc_K^{SLC}$.
\end{exam}

\begin{rema}\label{semi-lc}
Let $K$ be a continuum, planar or non-planar. By \cite[Theorem 2.4]{FitzGerald67}, if $M$ is a monotone decomposition, each element of $M$ is $T$-closed if and only if the hyperspace is semi-locally connected. Let $\Mc^{SLC}(K)$ denote all the monotone decompositions of $K$ with semi-locally connected hyperspace. By \cite[Theorem 2.7]{FitzGerald67}, there exists a unique  $\Dc_K^{SLC}\in \Mc^{SLC}(K)$ that is finer than the other elements of $\Mc^{SLC}(K)$. The decomposition $\Dc_K^{SLC}$ is called the \emph{core decomposition of $K$} with semi-locally connected hyperspace, and the hyperspace $\Dc_K^{SLC}$ under quotient topology is called the {\em semi-locally connected model for $K$}. For a continuum $K\subset\hat{\mathbb{C}}$, from Example \ref{slc-bad} we know that the result in Theorem \ref{invariance} does not hold if $\Dc^{PS}(K)$ is replaced by $\Dc^{SLC}(K)$.
\end{rema}

In the rest of this section, we want to represent the core decomposition of planar compacta with Peano hyperspace, in a way that follows FitzGerald-Swingle’s approach \cite{FitzGerald67}. A brief summary of this approach is recalled in Remark \ref{semi-lc}.  In the following we review further details concerning FitzGerald-Swingle's T-function and the core decomposition of a continuum with semi-locally connected hyperspace \cite{FitzGerald67}.

Given a compactum $K$ and  $A\subset K$, the set $T(A)$ consists of all $y\in K$ for which there do not exist both an open set $Q$ and a continuum $W$ such that $y\in Q\subset W\subset (K\setminus A)$. The set $T(A)$ is closed and contains $A$ for all $A\subset K$; moreover, it is connected whenever $A$ is. We say that $K$ is semi-locally connected at a point $x\in K$ provided that for every open subset $U$ of $K$ with $x\in U$, there exists an open subset $V$ such that $x\in V\subset U$ and $K\setminus V$ has only a finite number of components. 
By routine arguments one may check that $K$ is semi-locally connected at $x\in K$ if and only if $T(x)=x$  \cite[Lemma 1.4]{FitzGerald67}. Given a monotone decomposition $\Dc$ of $K$, the natural projection $\pi: K\rightarrow\Dc$ sending a point $x\in K$ to the only element of $\Dc$ containing $x$ even carries the previous result to the hyperspace, which is semi-locally connected at $\pi(x)$ if an only if $\pi(x)\in\Dc$ is $T$-closed in the sense that $T(\pi(x))=\pi(x)$  \cite[Theorem 2.4]{FitzGerald67}. From this it readily follows that, among the collection  $\Mc_K^{SLC}$ of all monotone decompositions whose elements are $T$-closed, there is a finest member $\Dc_K^{SLC}$, called the {\em core decomposition of $K$} with semi-locally connected hyperspace.

The question is: can we use such an approach to present the existence of $\Dc_K^{SLC}$, in which  $T$-function is replaced by $S$-function newly introduced in Definition \ref{S-function invariant} ? In other words, we are asking about a representation of the main theorems of \cite{LLY-2017} in the spirit of \cite{FitzGerald67}.

The answering of this question requires us to extend the $S$-function given in Definition \ref{S-function}, for single points only, to the cases of all continua $A\subset K$; moreover, we even require that $S(A)$ is always connected. By Theorem \ref{S-basic} we have $S(x)=\overline{R_K}[x]$ for all $x\in K$; by Theorem \ref{connected-fiber} the set $S(x)$ is a continuum. Therefore, we extend the $S$-function as follows.

\begin{deff}
Given a compactum $K$ and a continuum $A\subset K$, let $S(A)$ be the union of all those $S(x)$ with $x\in A$. Any continuum $A\subset K$ with $S(A)=A$ is called {\em $S$-closed}.
\end{deff}

Fix a compactum $K\subset\hat{\mathbb{C}}$. Then it is routine to verify that $S(A)$ is a continuum for every continuum $A\subset K$. When $K$ is connected, the result of \cite[Theorem 1.1]{Luo07} indicates that $K$ is locally connected if and only if $S(x)=\{x\}$ for all $x\in K$. On the other hand, by the result of \cite[Theorem 3]{LLY-2017} we see that $K$ is a Peano space if and only if $S(x)=\{x\}$ for all $x\in K$. Those results are stated in a way similar to the statement of \cite[Lemma 1.4]{FitzGerald67}. By \cite[Theorem 7]{LLY-2017}, we may follow the spirit of  \cite[Theorem 2.4]{FitzGerald67} and obtain the following criterion, which then ensures the existence of $\Dc_K^{PS}$.

\begin{theo}\label{S-key}
A monotone decomposition $\Dc$ of a compactum $K\subset\hat{\mathbb{C}}$ has a Peano hyperspace if and only if the elements of $\Dc$ are $S$-closed continua.
\end{theo}
\begin{proof}
Let $R_K$ be the Sch\"onflies relation and $\sim$ the Sch\"onflies equivalence. Let $\Dc_K$ be the parition of $K$ into $\sim$ classes. If all the elements of $\Dc$ are $S$-closed then $\Dc_K$ refines $\Dc$. Therefore, there is a projection $\pi: \Dc_K\rightarrow\Dc$. Combining this with \cite[Theorem 5]{LLY-2017}, we see that $\Dc$ under quotient topology is a Peano space. On the other hand, if the hyperspace of $\Dc$ is a Peano space the result of \cite[Theorem 6]{LLY-2017} implies that $\Dc$ is refined by $\Dc_K$ hence that every element of $\Dc$ is $S$-closed.
\end{proof}

For the moment it is unclear how useful the $S$-function is for a non-planar compactum $K$. In particular, we even do not know, for a non-planar compactum $K$,  whether $S(x)$ is connected for all $x\in K$; although this is true for $K\subset\hat{\mathbb{C}}$, see Theorem \ref{S-basic} and Theorem \ref{connected-fiber}.

To end this section, we compare the $T$-function and $S$-function in Theorem \ref{TS}, which implies that a monotone decomposition of a continuum into $S$-closed sub-continua necessarily has a semi-locally connected hyperspace. However, this hyperspace may not be locally connected.  non-locally connected The continuum $K\subset\mathbb{R}^3$ given in Example \ref{no-CD-ex} is not locally connected but semi-locally connected everywhere; moreover, it satisfies $S(x)=T(x)=\{x\}$ for all $x\in K$.

\begin{theo}\label{TS}
Let $K$ be a non-planar continuum. Then $T(x)\subset S(x)$ for any $x\in K$.
\end{theo}
\begin{proof}
Fix a point $y\notin S(x)$ we can find two disjoint open sets $U_x, U_y$ with $x\in U_x$ and $y\in U_y$ such that $K\setminus(U_x\cup U_y)$ has at most finitely many components, say $P_1,\ldots,P_n$, that intersect $\partial U_x$ and $\partial U_y$ both. Let $P_i^*$ be the union of $P_i$ with all the components of $K\cap \overline{U_y}$ that intersect $P_i$. Let $P_i^{**}$ be the union of all the components of $K\setminus(U_x\cup U_y)$ that intersect $P_i^*$. By the well known Boundary Bumping Theorem II~\cite[Theorem 5.6, p74]{Nadler92}, every component of $K\cap\overline{U_y}$ intersects at least one $P_i$ thus is contained in $P_i^*$. By the same reason, a component of $K\setminus(U_x\cup U_y)$ intersecting $\partial U_y$ intersects some $P_i^*$. Therefore, $\{P_i^{**}: 1\le i\le n\}$ is a finite cover of $K\setminus U_x$ by continua. This indicates that $K\setminus U_x$ may be covered by finitely many disjoint continua. Denote by $W$ the one containing $y$. Then $y$ lies in the interior of some $P_i^{**}$, which is disjoint from $U_x$. Setting $W=P_i^{**}$ and $V=W^o$. Then $V$ is an open set and $W$ is a continuum, satisfying $y\in V\subset W \subset  (K\setminus\{x\})$. This verifies that $y\notin T(x)$.
\end{proof}

\section{Planar Compacta $K$ with $\lambda(K)\in\mathbb{N}\cup\{\infty\}$ }\label{scale-examples}

In this section we recall two classical continua and two examples of planar continua $K$ from \cite{JLL16,LL-2018}. We will explicitly determine the function $\lambda_K: K\rightarrow\mathbb{N}\cup\{\infty\}$ for those $K$. In particular, Cantor's teepee~\cite[p.145]{SS-1995} provides an example of continuum $K$ with $\lambda_K(x)\equiv\infty$. See the following Figure \ref{teepee} for a simplified depiction.

\begin{figure}[ht]\vskip -0.45cm
\begin{center}

\begin{tikzpicture}[x=1.6cm,y=1cm,scale=0.8]
\foreach \i in {0,...,3}
{
\draw[gray,thick] (3.645,2) -- (0.09*\i,0);
\draw[gray,thick] (3.645,2) -- (0.54+0.09*\i,0);
\draw[gray,thick] (3.645,2) -- (1.62+0.09*\i,0);
\draw[gray,thick] (3.645,2) -- (2.16+0.09*\i,0);

\draw[gray,thick] (3.645,2) -- (4.86+0.09*\i,0);
\draw[gray,thick] (3.645,2) -- (5.4+0.09*\i,0);
\draw[gray,thick] (3.645,2) -- (6.48+0.09*\i,0);
\draw[gray,thick] (3.645,2) -- (7.02+0.09*\i,0);
}

\fill[black] (3.645,2) node[above]{$p$} circle (0.3ex);

\end{tikzpicture}

\end{center}\vskip -1.0cm
\caption{A simple representation of Cantor's Teepee.}\label{teepee}\vskip -0.25cm
\end{figure}
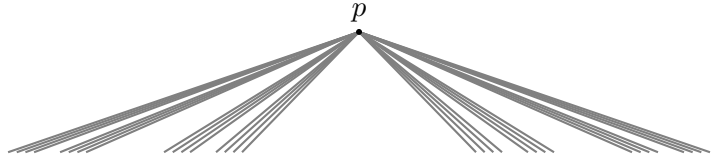

Generally, for any integer $n\ge1$, we can find a continuum $K_n\subset\mathbb{C}$ with $\lambda(K)=n$.  Denote by $\ell(K)$ the scale of non-local connectedness introduced in \cite{JLL16}. Then $\ell(K)=\lambda(K)$ and $\ell(K)<\lambda(K)$ are both possible. When $K$ is Cantor's teepee, we have $\ell(K)=1<\lambda(K)$. And we conjecture that $\ell(K)\le\lambda(K)$ holds for all compacta in the plane.

\begin{exam}[{\bf Witch's Broom}]\label{relations}
Let $K\subset\mathbb{C}$ be the Witch's Broom \cite[p.84, Figure 5.22]{Nadler92}. Then $\lambda_K(x)=0$ for all $x\notin\mathbb{R}$; and $\lambda_K(x)=1$ for $x\in[0,1]$.
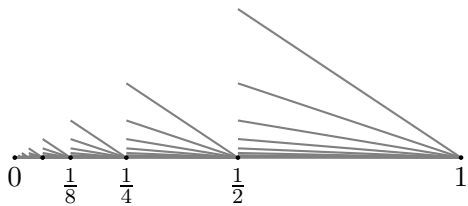
\begin{figure}[ht]
\vskip -0.15cm\begin{center}
\begin{tikzpicture}[x=1.5cm,y=1cm,scale=0.618]
\foreach \i in {0,...,8}
{
\draw[gray,thick] (6.4,0) -- (3.2,3.2/2^\i);
\draw[gray,thick] (3.2,0) -- (1.6,1.6/2^\i);
\draw[gray,thick] (1.6,0) -- (0.8,0.8/2^\i);
\draw[gray,thick] (0.8,0) -- (0.4,0.4/2^\i);
\draw[gray,thick] (0.4,0) -- (0.2,0.2/2^\i);
\draw[gray,thick] (0.2,0) -- (0.1,0.1/2^\i);
\draw[gray,thick] (0.1,0) -- (0.05,0.05/2^\i);
\draw[gray,thick] (0.05,0) -- (0.025,0.025/2^\i);
}
\draw[gray,thick] (0,0) -- (6.4,0);
\fill[black] (6.4,0) node[below]{$1$} circle (0.3ex);
\fill[black] (3.2,0) node[below]{$\frac12$} circle (0.3ex);
\fill[black] (1.6,0) node[below]{$\frac14$} circle (0.3ex);
\fill[black] (0.8,0) node[below]{$\frac18$} circle (0.3ex);
\fill[black] (0.4,0)  circle (0.3ex);
\fill[black] (0,0) node[below]{$0$} circle (0.3ex);
\end{tikzpicture}
\end{center}
\vskip -1.0cm
\caption{A finite approximation of the Witch's Broom.}\label{broom}
\vskip -0.25cm
\end{figure}
Thus we have $\ell(K)=\lambda(K)=1$.
\end{exam}

\begin{exam}[{\bf Cantor's Comb}]\label{cantor-comb}
Let $\Kc\subset[0,1]$ be Cantor's ternary set. Let $K$ be the union of $\Kc\times[0,1]\subset\mathbb{C}$ and $\{t+{\bf i}: 0\le t\le 1\}$, which may be called Cantor's Comb. Then $\lambda_K(x)=0$ for all $x\notin K\setminus(\Kc\times[0,1])$; moreover, $\lambda_K(x)=1$ for other $x\in K$. Here we have $\lambda(K)=\ell(K)=1$.
\begin{figure}[ht]
\begin{center}\vskip -0.25cm

\begin{tikzpicture}[x=1cm,y=1cm,scale=1.0]

\draw[very thick] (0,3) -- (3,3);

\foreach \i in {0,...,3}
{
    \draw[gray,thick] (3*\i/27,0) -- (3*\i/27,3);
    \draw[gray,thick] (6/9+3*\i/27,0) -- (6/9+3*\i/27,3);
     \draw[gray,thick] (2+3*\i/27,0) -- (2+3*\i/27,3);
    \draw[gray,thick] (2+6/9+3*\i/27,0) -- (2+6/9+3*\i/27,3);
}
\fill[black] (0,3) node[left]{${\bf i}$} circle (0.3ex);
\fill[black] (3,3) node[right]{$1+{\bf i}$} circle (0.3ex);


\end{tikzpicture}
\end{center}\vskip -1.0cm
\caption{A finite approximation of Cantor's Comb.}\label{comb}
\end{figure}
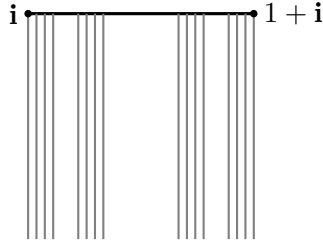
\end{exam}

\begin{exam}\label{scale=2}
Let $K$ be Cantor's Comb given in Example \ref{cantor-comb}. Let $K_2$ be the union of $K$ and its image $g(K)$ under $g(z)=1+{\bf i}-z{\bf i}$, as depicted in the left part of Figure \ref{combs}. Then $\lambda_{K_2}(x)=2$ for all $x\in \Kc\times[0,1]$, and $\lambda_{K_2}(x)=0$ for all $x\in g(K)\setminus g(\Kc\times[0,1])$; for the other cases, we have $\lambda_{K_2}(x)=1$.
\begin{figure}[ht]
\begin{center}
\vskip -0.0cm

\begin{tikzpicture}[x=1cm,y=1cm,scale=1.0]

\foreach \i in {0,...,3}
{
    \draw[thick, gray] (3*\i/27,0) -- (3*\i/27,3);
    \draw[thick, gray] (6/9+3*\i/27,0) -- (6/9+3*\i/27,3);
     \draw[thick, gray] (2+3*\i/27,0) -- (2+3*\i/27,3);
    \draw[thick, gray] (2+6/9+3*\i/27,0) -- (2+6/9+3*\i/27,3);
}
\draw[black,very thick] (0,3) -- (3,3);
\fill[black] (0,3) circle (0.3ex); \fill[black] (3,3) circle (0.3ex);

\foreach \i in {0,...,3}
{
    \draw[thick, gray] (0,3+3*\i/27) -- (3,3+3*\i/27);
    \draw[thick, gray] (0,3+6/9+3*\i/27) -- (3,3+6/9+3*\i/27);
     \draw[thick, gray] (0,3+2+3*\i/27) -- (3,3+2+3*\i/27);
    \draw[thick, gray] (0,3+2+6/9+3*\i/27) -- (3,3+2+6/9+3*\i/27);
}
\draw[black,very thick] (3,3) -- (3,6);
\fill[black] (3,3) circle (0.3ex); \fill[black] (3,6) circle (0.3ex);


\foreach \i in {0,...,3}
{
    \draw[thick, gray] (7+3*\i/27,0) -- (7+3*\i/27,3);
    \draw[thick, gray] (7+6/9+3*\i/27,0) -- (7+6/9+3*\i/27,3);
     \draw[thick, gray] (7+2+3*\i/27,0) -- (7+2+3*\i/27,3);
    \draw[thick, gray] (7+2+6/9+3*\i/27,0) -- (7+2+6/9+3*\i/27,3);
}
\draw[black,very thick] (7,3) -- (7+3,3);
\fill[black] (7,3) circle (0.3ex); \fill[black] (7+3,3) circle (0.3ex);

\foreach \i in {0,...,3}
{
    \draw[thick, gray] (7,3+3*\i/27) -- (7+3,3+3*\i/27);
    \draw[thick, gray] (7,3+6/9+3*\i/27) -- (7+3,3+6/9+3*\i/27);
     \draw[thick, gray] (7,3+2+3*\i/27) -- (7+3,3+2+3*\i/27);
    \draw[thick, gray] (7,3+2+6/9+3*\i/27) -- (7+3,3+2+6/9+3*\i/27);
}
\draw[black,very thick] (10,3) -- (10,6);
\fill[black] (10,3) circle (0.3ex); \fill[black] (10,6) circle (0.3ex);

\foreach \i in {0,...,3}
{
    \draw[thick, gray] (3+7+3*\i/27,3+0) -- (3+7+3*\i/27,3+3);
    \draw[thick, gray] (3+7+6/9+3*\i/27,3+0) -- (3+7+6/9+3*\i/27,3+3);
     \draw[thick, gray] (3+7+2+3*\i/27,3+0) -- (3+7+2+3*\i/27,3+3);
    \draw[thick, gray] (3+7+2+6/9+3*\i/27,3+0) -- (3+7+2+6/9+3*\i/27,3+3);
}
\draw[black,very thick] (10,6) -- (13,6);
\fill[black] (10,6) circle (0.3ex); \fill[black] (10+3,6) circle (0.3ex);

\end{tikzpicture}

\end{center}\vskip -0.5cm
\caption{(Left). The union of the Cantor's Comb with one of its copies. (Right). $K_3$.}\label{combs}\vskip -0.25cm
\end{figure}
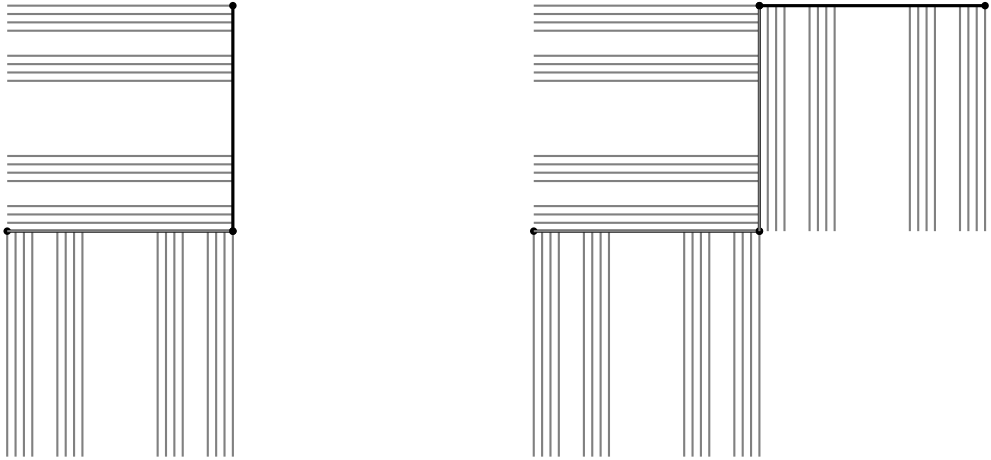
Given the continuum $K_2$, we may put $K_3=K_2\cup\left(K+1+{\bf i}\right)$, which is a continuum with $\lambda(K_3)=3$. See the right part of Figure \ref{combs}. This construction may be done indefinitely, to obtain a continuum $K_n$ with $\lambda(K_n)=n$ for $n\ge4$. The general formula 
\[K_{2m+1}=K_{2m}\bigcup\left(K+m+m{\bf i}\right)\quad\text{and}\quad K_{2m+2}=K_{2m+1}\bigcup\left(g(K)+m+m{\bf i}\right)\]
are set for all $m\ge1$. Given such a continuum $K_n$, its lambda function  is nearly immediate, as soon as the lambda function of $K_{n-1}$ is determined. Actually, the containments $K\subset K_2\subset K_3\subset\cdots\subset K_{n-1}\subset K_n$ are clear. Moreover, it is routine to check that $K_{n-1}$ is a member of the core decomposition $\Dc_{K_n}^{PS}$, which indicates that $\lambda_{K_n}(x)=1+\lambda_{K_{n-1}}(x)$ for all $x\in K_{n-1}\subset K_n$. \end{exam}

The last example is a response to a well known result by Moore \cite{Moore25-a}, concerning a simple monotone decomposition of a continuum such that the hyperspace is a Peano continuum.

Given a continuum $K$, planar or nonplanar, let  $K_{NLC}$ be the collection of all the points $x\in K$ at which $K$ is not locally connected. Then $\overline{K_{NLC}}$ gives rise to a monotone decomposition $\Dc_{K,LC}$ whose elements are either a singleton contained in $K\setminus\overline{K_{NLC}}$ or a component of $\overline{K_{NLC}}$. By Moore's result \cite{Moore25-a} (See also \cite[p.247, \S 49, VI, Theorem 3]{Kuratowski68}), the hyperspace $\Dc_{K,LC}$ is a Peano continuum. To be brief, we may refer to $\Dc_{K,LC}$ as {\em Moore's decomposition of $K$}.

In the following we construct a concrete continuum $K\subset\mathbb{C}$ such that $\lambda_K^{-1}(0)$ contains a non-degenerate component of $\overline{K_{NLC}}$. This implies that the core decomposition of $\Dc_K^{PS}$ is strictly finer than $\Dc_{K,LC}$. On the other hand, Example \ref{no-CD-ex}
gives a non-planar continuum $K\subset\mathbb{R}^3$ such that there are uncountably many monotone decompositions whose hyperspaces are Peano continua; moreover, those decompositions are strictly finer than Moore's decomposition $\Dc_{K,LC}$.

\begin{exam}\label{NLC-point-CD}
Let $f_1(z)=\frac{1}{2}z$ and $f_2(z)=\frac{1}{2}(z+{\bf i})$. Then $E=\{t{\bf i}: 0\le t\le 1\}\subset\mathbb{C}$ is the unique non-empty compact set satisfying $E=f_1(E)\cup f_2(E)$ \cite{Hutchinson81}.
Following Barnsley's idea, we may choose a continuum $A$ as the {\em condensation set}, which is the union of a copy of Cantor's comb and the broken line connecting $1$ through $2,2+{\bf i}$ to $1+{\bf i}$. See the right part of Figure \ref{nlc-points}.
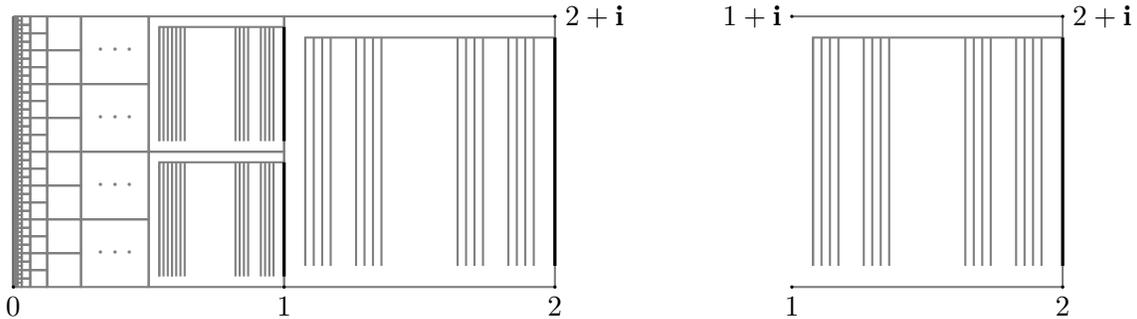
\begin{figure}[ht]
\begin{center}

\begin{tikzpicture}[x=0.25cm,y=0.25cm,scale=0.45]

\draw[thick] (0,0) -- (0,32);
\draw[gray,thick] (64,0) -- (64,32) -- (32,32) -- (32,0) -- (64,0);

\foreach \j in {0,1}
{
    \draw[gray, thick] (32,\j*16) -- (32,16+\j*16) -- (16,16+\j*16) -- (16,\j*16) --(32,\j*16);
}

\foreach \j in {0,...,3}
{
    \draw[gray, thick] (16,\j*8) -- (16,8+\j*8) -- (8,8+\j*8) -- (8,\j*8) --(16,\j*8);
}

\foreach \j in {0,...,7}
{
    \draw[gray, thick] (8,\j*4) -- (8,4+\j*4) -- (4,4+\j*4) -- (4,\j*4) --(8,\j*4);
}

\foreach \j in {0,...,15}
{
    \draw[gray, thick] (4,\j*2) -- (4,2+\j*2) -- (2,2+\j*2) -- (2,\j*2) --(4,\j*2);
}

\foreach \j in {0,...,31}
{
    \draw[gray, thick] (2,\j*1) -- (2,1+\j*1) -- (1,1+\j*1) -- (1,\j*1) --(2,\j*1);
}

\foreach \j in {0,...,63}
{
    \draw[gray, thick] (1,\j*1/2) -- (1,1/2+\j*1/2) -- (1/2,1/2+\j*1/2) -- (1/2,\j*1/2) --(1,\j*1/2);
}

\foreach \j in {0,...,127}
{
    \draw[gray, thick] (1/2,\j*1/4) -- (1/2,1/4+\j*1/4) -- (1/4,1/4+\j*1/4) -- (1/4,\j*1/4) --(1/2,\j*1/4);
}

\foreach \j in {0,...,255}
{
    \draw[gray, thick] (1/4,\j*1/8) -- (1/4,1/8+\j*1/8) -- (1/8,1/8+\j*1/8) -- (1/8,\j*1/8) --(1/4,\j*1/8);
}

\fill[black] (0,0) node[below]{$0$} circle (0.3ex);
\fill[black] (32,0) node[below]{$1$} circle (0.3ex);
\fill[black] (64,0) node[below]{$2$} circle (0.3ex);
\fill[black] (64,32) node[right]{$2+{\bf i}$} circle (0.3ex);

\draw[gray,thick] (92,0) -- (124,0) -- (124,32) -- (92,32);

\fill[black] (92,0) node[below]{$1$} circle (0.3ex);
\fill[black] (124,0) node[below]{$2$} circle (0.3ex);
\fill[black] (124,32) node[right]{$2+{\bf i}$} circle (0.3ex);
\fill[black] (92,32) node[left]{$1+{\bf i}$} circle (0.3ex);

\foreach \i in {0,...,3}
{
    \draw[gray,thick] (92+\i+2.5,29.5) -- (92+\i+2.5,2.5);
    \draw[gray,thick] (92+\i+2.5+6,29.5) -- (92+\i+2.5+6,2.5);
    \draw[gray,thick] (92+\i+2.5+18,29.5) -- (92+\i+2.5+18,2.5);
    \draw[gray,thick] (92+\i+2.5+24,29.5) -- (92+\i+2.5+24,2.5);
}
\draw[gray,thick] (92+2.5,2.5) -- (92+2.5,29.5) -- (124,29.5) -- (124,2.5);
\draw[very thick]  (124,29.5) -- (124,2.5);

\foreach \i in {0,...,3}
{
    \draw[gray,thick] (32+\i+2.5,29.5) -- (32+\i+2.5,2.5);
    \draw[gray,thick] (32+\i+2.5+6,29.5) -- (32+\i+2.5+6,2.5);
    \draw[gray,thick] (32+\i+2.5+18,29.5) -- (32+\i+2.5+18,2.5);
    \draw[gray,thick] (32+\i+2.5+24,29.5) -- (32+\i+2.5+24,2.5);
}
\draw[gray,thick] (32+2.5,2.5) -- (32+2.5,29.5) -- (64,29.5) -- (64,2.5);
\draw[very thick]  (64,29.5) -- (64,2.5);

\foreach \i in {0,...,3}
{
    \draw[gray,thick] (34.5/2+\i/2,29.5/2) -- (34.5/2+\i/2,2.5/2);
    \draw[gray,thick] (37.5/2+\i/2,29.5/2) -- (37.5/2+\i/2,2.5/2);
    \draw[gray,thick] (52.5/2+\i/2,29.5/2) -- (52.5/2+\i/2,2.5/2);
    \draw[gray,thick] (58.5/2+\i/2,29.5/2) -- (58.5/2+\i/2,2.5/2);
}
\draw[gray,thick] (34.5/2,2.5/2) -- (34.5/2,29.5/2) -- (64/2,29.5/2) -- (64/2,2.5/2);
\draw[very thick]  (64/2,29.5/2) -- (64/2,2.5/2);

\foreach \i in {0,...,3}
{
    \draw[gray,thick] (34.5/2+\i/2,16+29.5/2) -- (34.5/2+\i/2,16+2.5/2);
    \draw[gray,thick] (37.5/2+\i/2,16+29.5/2) -- (37.5/2+\i/2,16+2.5/2);
    \draw[gray,thick] (52.5/2+\i/2,16+29.5/2) -- (52.5/2+\i/2,16+2.5/2);
    \draw[gray,thick] (58.5/2+\i/2,16+29.5/2) -- (58.5/2+\i/2,16+2.5/2);
}
\draw[gray,thick] (34.5/2,16+2.5/2) -- (34.5/2,16+29.5/2) -- (64/2,16+29.5/2) -- (64/2,16+2.5/2);
\draw[very thick]  (64/2,16+29.5/2) -- (64/2,16+2.5/2);

\foreach \i in {0,...,3}
{
\fill[gray] (16,4+8*\i) node[left]{\large $\cdots$};
}

\end{tikzpicture}

\end{center}\vskip -1.0cm
\caption{A finite approximation of $K$.}\label{nlc-points}
\end{figure}
There is a unique compact set $K$ satisfying $K=A\cup f_1(K)\cup f_2(K)$ \cite[p.91, Theorem 9.1]{Barnsley93}. The left part of Figure \ref{nlc-points} is a simplified depiction of $A\cup f_1(A)\cup f_2(A)$ and indicates the inductive approach to define $K$. Clearly, $K$ is a continuum. If we denoted by $K_{NLC}$ the set of points $x\in K$ at which $K$ is not locally connected, then $E$ is a component of $\overline{K_{NLC}}$ and hence is an element of Moore's decomposition $\Dc_{K,LC}$. However, the following equation holds:
\begin{equation}\label{moore}
\lambda_K(x)=\left\{\begin{array}{ll}1& x\in K_{NLC}\\ 0& otherwise\end{array}\right.
\end{equation}
In particular, $\lambda_K(x)=0$ for all $x\in E$. Since the hyperspace $\Dc_{K,LC}$ is also a Peano continuum, we immediately see that the core decomposition $\Dc_K^{PS}$ strictly refines $\Dc_{K,LC}$.
\end{exam}

\begin{proof}[\bf Proof for Equation \ref{moore}] Our argument uses the following Lemma \ref{union}, which itself is of some interest. In this lemma, the two hyperspaces $\Dc_{K_1}^{PS}$ and $\Dc_{K_2}^{PS}$ are embedded into the hyperspace $\Dc_{K_1}^{PS}\cup\Dc_{K_2}^{PS}$ of $K$, given by the monotone decomposition $\Dc_{K_1}^{PS}\cup\Dc_{K_2}^{PS}$. Since $\Dc_{K_1}^{PS}$ and $\Dc_{K_2}^{PS}$ are both Peano spaces and since they intersect at finitely many points, their union is again a Peano space. Therefore, we have $\Dc_K^{PS}=\Dc_{K_1}^{PS}\cup\Dc_{K_2}^{PS}$.
\begin{lemma}\label{union}
Let $K_1,K_2\subset\mathbb{C}$ be compacta with $K_1\cap K_2$ a finite set. Let $K=K_1\cup K_2$. If for every $x\in K_1\cap K_2$ the element $D_i(x)$ of $\Dc_{K_i}^{PS}$ for $i=1,2$ equals $\{x\}$ then  $\Dc_K^{PS}=\Dc_{K_1}^{PS}\cup\Dc_{K_2}^{PS}$. In particular, we have $\{x\}\in \Dc_K^{PS}$ for all $x\in K_1\cap K_2$.
\end{lemma}
For any point $x$ lying in $K\setminus (K_{NLC}\cup E)$, it is clear that there is a small enough number $r_x>0$ such that the intersection of $K$ and the disk $D_x$, centered at $x$ with radius $r_x$, is either a segment (horizontal or vertical), or the union of two segments that has a $T$-shape or an up-side-down $T$-shape. For such an $x$, we apply Lemma \ref{union} to obtain that $D(x)=\{x\}$, indicating that $\lambda_K(x)=0$.

If $x\in K_{NLC}$ then it lies on some vertical segment, a ``tooth'' of some small copy of Cantor's comb. Directly we can check that the element $D(x)$ of $\Dc_K^{PS}$ containing $x$ is exactly that segment. This verifies $\lambda_K(x)=1$. Actually, we may put $A_1=A$ and $A_2=A\cup f_1(A)\cup f_2(A)$. Generally, for $n\ge3$ we put $A_n=A\cup f_1(A_{n-1})\cup f_2(A_{n-1})$. Then $\{A_n: n\ge1\}$ is an increasing sequence of continua whose union is equal to $K\setminus E$. Let $K_n=\overline{K\setminus A_n}$. Then $\{K_n: n\ge1\}$ is an infinite sequence of continua that decreasingly converge to $E$.
Moreover, the intersection $K_n\cap A_n$ is a finite set for all $n\ge1$. By Lemma \ref{union} we see that $\left\{\Dc_{A_n}^{PS}: n\ge1\right\}$ is an increasing sequence of sub-collections of $\Dc_K$.
A closer look at $A$ will lead us to the observation that every element of $\Dc_A^{PS}$ either is a single point off the teeth or coincides with a whole tooth. Combining this with Lemma \ref{union}, we can infer  that $\lambda_K(x)=1$ for all points $x$ that stay on a tooth for some small copy of Cantor's comb.

From the above arguments we can infer that for every $x\in (K\setminus E)$ the element of $\Dc_K^{PS}$ containing $x$, denoted as $D(x)$, equals the fiber of $\overline{R_K}$ at $x$, denoted as  $\overline{R_K}(x)$.

Lastly, we turn to the computation of $\overline{R_K}(x)$ for $x\in E$. The closed relation $\overline{R_K}$ is characterized in Theorem \ref{closure}, by which we may directly check that if $x\in E$ then the fiber $\overline{R_K}(x)$ equals the singleton $\{x\}$.
Therefore $\left\{\overline{R_K}(x): x\in K\right\}$ is a monotone decomposition, which necessarily equals the decomposition given by the Sch\"onflies equivalence on $K$. The Sch\"onflies equivalence is defined to be the smallest closed equivalence containing $R_K$. By \cite[Theorem 5 and 6]{LLY-2017} we have $\left\{\overline{R_K}(x): x\in K\right\}=\Dc_K^{PS}$. This verifies that $D(x)=\{x\}$ thus $\lambda_K(x)=0$ for $x\in E$.
\end{proof}

\bibliographystyle{plain}

\end{document}